\documentclass{amsart}
\usepackage[outer=1.25in, inner=1.25in, top=1.25in, bottom=1.25in]{geometry}
\newtheorem{theorem}{Theorem}[section]
\newtheorem{definition}[theorem]{Definition}
\newtheorem{proposition}[theorem]{Proposition}

\newtheorem{lemma}[theorem]{Lemma}

\theoremstyle{remark}
\newtheorem*{remark}{Remark}

\usepackage{hyperref,amssymb}
\begin{document}

	\title[minimal degree and base size]{On the minimal degree and base size of finite primitive groups}

	\author[F. Mastrogiacomo]{Fabio Mastrogiacomo}
	\address{Fabio Mastrogiacomo, Dipartimento di Matematica ``Felice Casorati", University of Pavia, Via Ferrata 5, 27100 Pavia, Italy} 
	\email{fabio.mastrogiacomo01@universitadipavia.it}

	\maketitle
	\begin{abstract}
		Let $G$ be a finite permutation group acting on $\Omega$. A base for $G$ is a subset $B \subseteq \Omega$ such that the pointwise stabilizer $G_{(B)}$ is the identity. The base size of $G$, denoted by $b(G)$, is the cardinality of the smallest possible base. The minimal degree of $G$, denoted by $\mu(G)$, is the smallest cardinality of the support of a non trivial element of $G$. In this paper, we establish a new upper bound for $b(G)$ when $G$ is primitive, and subsequently prove that if $G$ is a primitive group different from the Mathieu group of degree $24$, then $\mu(G)b(G)\leq n \log n$, where $n$ is the degree of $G$. This bound is best possible, up to a multiplicative constant.
	\end{abstract}
	\section{introduction}
	Let $G$ be a finite permutation group acting on a set $\Omega$. The \textit{\textbf{minimal degree}} of $G$, denoted by $\mu(G)$, is the smallest number of elements of $\Omega$ that are moved by any non-identity element of $G$. A base for $G$ is a subset $B \subseteq \Omega$ such that $G_{(B)}=1$. The \textit{\textbf{base size}} of $G$, denoted by $b(G)$, is the smallest cardinality of a base. These two invariants are among the most important ones studied in permutation group theory. In particular, the base size plays a crucial role in the investigation of many properties of the group $G$. Many authors have consequently studied this invariant and provided  upper bounds for it. On one side, the famous Pyber's conjecture, now a theorem (see \cite{BS} and \cite{DHM}), gives an upper bound in term of the size of the group and the degree. On the other side, some upper bound has been given for the base size depending only on the degree of the group. Some results have also been obtained regarding the minimal degree, as seen in \cite{BG} and \cite{KPS}. However, it appears that these two ingredients  have not been studied together. Considering, for example, their product, the only known result about it is the following elementary lower bound, which can be found in \cite[Excercise~$3.3.7$]{dixon_mortimer}. If $G$ is a transitive group of degree $n$, then
	\[
		n \leq b(G)\mu(G).
	\]
	The main topic explored in this paper is the existence of a function $f :\mathbb{N} \to \mathbb{R}$ satisfying $\mu(G) b(G) \leq f(n)$.\\
	As an initial approach to this problem, one might consider leveraging existing results on the base size of primitive groups. For example,  using the classification of finite simple groups, Liebeck \cite[Theorem~$1$]{Lieb} has proved that if $G$ is a primitive group of degree $n$, which is not large-base (see Section~\ref{sec2:sub3} for the definition of large-base group), then $b(G) \leq 9 \log n$ (in this paper, all the logarithms are taken in base $2$).  Subsequently, this result was improved by Moscatiello and Roney-Dougal in \cite[Theorem~$5$]{MRD}; their result proves that, if $G$ is not large-base, then $b(G) \leq 2+\log n$. Using these results, one may establish inequalities of the form 
	\[
		\mu(G) b(G) \leq 9n\log n,
	\]
	or 
	\[
		\mu(G) b(G) \leq n(2+\log n).
	\]
	However, bounds of this form are far from sharp and do not apply to large-base groups. The main result of this paper presents an improvement on these earlier inequalities and provides a uniform bound.
	\begin{theorem} \label{mainthm:1}
		Let $G$ be a primitive permutation group of degree $n$. If $G$ is not the Mathieu group $M_{24}$ in its action of degree $24$, then
		\[
			\mu(G) b(G) \leq n\log n.
		\]
	\end{theorem}
	Observe that if $G = M_{24}$ is the Mathieu group of degree $24$, then $b(G) = 7$ and $\mu(G) = 16$, so $b(G) \mu(G) = 112$, while $n \log n \approx 110$.\\
	Our proof of Theorem~\ref{mainthm:1} heavily relies on bounds on the base size of primitive groups, since more information is known about it. In particular, we found it necessary to derive new bounds for the base size of certain primitive groups. The bound we have established to prove Theorem~\ref{mainthm:1}, which is of independent interest, improves upon the result of Moscatiello and Roney-Dougal and is stated in the following theorem.
	\begin{theorem}\label{mainthm:2}
		Let $G$ be a finite primitive group of degree $n$, and suppose that $G$ is not large-base. Suppose moreover that $G$ is neither an almost simple group with socle $G_0$ acting on the set $\Omega$, with $(G_0,\Omega)$ as in Table \ref{tab:md}, nor a subgroup of $\mathrm{AGL}_d(q)$, with $q=2,3$. Then,
		\[
			b(G) \leq \frac{1}{2}\log n +6.
		\]
	\end{theorem}
	The bound in Theorem~\ref{mainthm:2} is asymptotically best possible. For instance, consider the affine group $G = \mathrm{AGL}_d(4)$. The base size of this group is $d+1$\footnote{In general, $b(\mathrm{AGL}_d(q)) = d+1$. The stabilizer of a point, that we may suppose to be the zero vector, is $\mathrm{GL}_d(q)$ acting on the non-zero vectors, and a base for this action is a basis for the vector space.}, while $n = 4^d$, and $d+1 \sim (\log n)/2+6$. Although the constant $6$ could potentially be reduced in certain cases, it is sufficient for our purposes to keep it as $6$.\\ 
	Regarding the sharpness of the bound in Theorem~\ref{mainthm:1}, this is asymptotically optimal up to a multiplicative constant. Indeed, there are infinitely many primitive groups $G$ of degree $n$ such that $\mu(G) b(G) \geq (n \log n)/2$. For example, take $G = \mathrm{AGL}_d(2)$ in its natural action of degree $2^n$. Here, we have $\mu(G) = n/2$ (see for example \cite[Excercise~$3.3.3$]{dixon_mortimer}) and $b(G) = 1+\log n$, so that $\mu(G) b(G) > (n \log n)/2$. \\	
	
	The proof of Theorem~\ref{mainthm:2} is carried out through a case by case analysis using the O'Nan-Scott Theorem.\\
	For what regards Theorem~\ref{mainthm:1}, if $G$ is a primitive group of degree $n$ that is not among the exceptions listed in Theorem~\ref{mainthm:2}, the inequality $\mu(G) b(G) \leq n \log n$ will easily follow from Theorem~\ref{mainthm:2}. In cases where $G$ is one of the exceptions listed in Theorem~\ref{mainthm:2}, the result is established by examining the exceptions listed in Table~\ref{tab:md}.\\

	The paper is organized as follows. In Section~\ref{sec1} we prove Theorem~\ref{mainthm:2}. In particular, in Subsection~\ref{sec1:sub1}, we prove Theorem~\ref{mainthm:2} for almost simple groups. In Subsection~\ref{sec1:sub2} we prove Theorem~\ref{mainthm:2} for diagonal groups. In Subsection~\ref{sec1:sub3} we prove Theorem~\ref{mainthm:2} for affine groups. In Subsection~\ref{sec1:sub4} we prove Theorem~\ref{mainthm:2} for product type groups. \\
	Section~\ref{sec2} is devoted to the proof of Theorem~\ref{mainthm:1}. In particular, in Subsection~\ref{sec2:sub1} we prove Theorem~\ref{mainthm:1} for affine groups with underlying field with $2$ or $3$ elements. In Subsection~\ref{sec2:sub2} we prove Theorem~\ref{mainthm:1} for almost simple groups having socle $G_0$ acting on $\Omega$, with $(G_0,\Omega)$ as in Table~\ref{tab:md}. In Subsection~\ref{sec2:sub3} we prove Theorem~\ref{mainthm:1} for large-base groups.
		\begin{table}
			\[
			\begin{array}{lllll} \hline
				G_0 & \Omega & \mbox{Conditions} \\ \hline
				{\rm PSL}_{d}(q) & P_1  & d \geq 3, q=2,3   \\
				\hline  
				{\rm PSp}_d(q) & P_1 & d \geq 6, q=2,3 \\
				&[G: N_G(\mathrm{GO}_d^\pm(q))] & d \geq 6, q=2 \\
				\hline 
				\mathrm{P}\Omega_d^\varepsilon(q) & S_1 & d \geq 8, q=2,3 		\\	
				 & N_1 &  d \geq 8, q=2,3 		\\	
				 \hline 
				\mathrm{P}\Omega_d(q) & S_1 & d \geq 7, q=3 \\
				& N_{d-1} & d \geq 7, q=3\\
				\hline
			\end{array}
			\]
			\caption{The subspace actions  of Theorem \ref{mainthm:2}}
			\label{tab:md}
		\end{table}\\

	\noindent {\bf Acknowledgements.} {\sl This work was conducted during a visit at the Alfréd Rényi Institute of Mathematics in Budapest. The author expresses gratitude to Attila Maróti for numerous helpful discussions on the topic and to the Rényi Institute for its hospitality.\\
	\noindent The author is a member of the GNSAGA INdAM research group and kindly acknowledges their support.}
	\section{Proof of Theorem~\ref{mainthm:2}}\label{sec1}
	In this section, we prove Theorem~\ref{mainthm:2}. We divide the proof according to the O'Nan-Scott theorem.
	\subsection{Almost simple groups}\label{sec1:sub1}
	Let $G$ be an almost simple group with socle $G_0$ in a primitive action of degree $n$. From now on, we will use the following notations for actions of classical groups.\\

	\noindent {\bf Notation.} {\sl
		Let $G$ be a classical group, with natural module $V$. We write $P_k$ to refer to the set of all $k$-dimensional subspaces of $V$. Moreover, we write $S_k$ for a $G$-orbit of totally singular subspaces of $V$ of dimension $k$, and $N_k$ for a $G$-orbit of non-degenerate or
		non-singular subspaces of $V$ of dimension $k$. For the orthogonal groups of even dimension, let $W$ be a space in the orbit if $dk$ is even, and the orthogonal complement of such a space if $dk$ is odd. Then we write $N_k^\epsilon$, with
		$\epsilon\in\{+,-\}$ to indicate that the restriction of the form to $W$ is of type $\epsilon$. \\
		We let $e_1,\ldots,e_d$ be the canonical basis of $V=\mathbb{F}_q^d$.\\
	}

	Before embarking on the proof of Theorem~\ref{mainthm:2}, we first recall the definition of subspace subgroup.
	\begin{definition}\label{defintionsubspacesubgrp}
		{\rm
			Let $G$ be an almost simple classical group over $\mathbb{F}_q$, where $q=p^f$ and $p$ is prime, with socle $G_0$ and associated natural
			module $V$. A subgroup $H$ of $G$ not containing $G_0$ is a \textit{\textbf{subspace subgroup}} if, for each maximal
			subgroup $M$ of $G_0$ containing $H \cap G_0$, one of the following holds:
			\begin{enumerate} 
				\item\label{definition6.2:1} $M$ is the stabilizer in $G_0$ of a proper nonzero subspace $U$ of $V$, where $U$ is totally
				singular, non-degenerate, or, if $G_0$ is orthogonal and $p = 2$, a nonsingular $1$-space ($U$ can be any subspace if $G_0 = \mathrm{PSL}(V )$).
				\item\label{definition6.2:2} $G_0 = \mathrm{Sp}'_{d} (q)$, $p = 2$, and 
				$M = \mathrm{GO}^\pm_{d} (q)$.
			\end{enumerate}
		}
	\end{definition}
	A \textbf{\textit{subspace action}} of the classical group $G$ is the action of $G$ on the coset space $[G : H]$,
	where $H$ is a subspace subgroup of $G$.
	
	\begin{definition}\label{definition6.3}
		{\rm
			A transitive action of $G$ on a set $\Omega$ is said to be standard if, up to
			equivalence of actions, one of the following holds:
			\begin{enumerate}
				\item $G_0 = \mathrm{Alt}(m)$ and $\Omega$ is an orbit of subsets or uniform partitions of $\{1, \ldots , m\}$.
				\item $G$ is a classical group in a subspace action.
			\end{enumerate}
		}
	\end{definition}
	For an almost simple primitive permutation group in a non-standard action, the base
	size is bounded by an absolute constant. This was conjectured by Cameron and Kantor
	(see~\cite{Cam92,CK97}) and then settled in the aﬃrmative by Liebeck and Shalev in~\cite[Theorem~1.3]{LS99}. The constant was then made explicit in subsequent work~\cite{BGS11,BLS09,BOW10,Bur07,Bur2018}. The following theorem summarizes these results.
	
	\begin{theorem}\label{thrm:6.4} Let $G$ be a finite almost simple group in a primitive faithful non-standard
		action with socle $G_0$. Then, $b(G) \le 7$, with equality if and only if $G$ is the Mathieu group $M_{24}$ in its natural action of degree $24$. 
	\end{theorem}
	With this strong result, we can easily prove Theorem~\ref{mainthm:2} for almost simple groups in a non-standard action. 
	\begin{proposition}
		Let $G$ be an almost simple group, equipped with a non-standard action of degree $n$. Then, $b(G) \leq (\log n)/2 + 6$.
	\end{proposition}
	\begin{proof}
		We have $b(G) \leq 7 \leq (\log n)/2+6$ since $n \geq 4$.
	\end{proof}
	
	We now take into account the standard actions of almost simple groups. To do this, we shall divide case by case, depending on the simple group $G_0 = \mathrm{soc}(G)$.
	\subsubsection{The socle of $G$ is an alternating group}
	Let $G$ be an almost simple group with socle $G_0 = A_m$. Since we are assuming that $G$ is not a large-base group, the actions we need to consider are those on uniform partitions of the set $\{1,\dots,m\}$. Moreover, since $b(G) \leq b(S_m)$, we can suppose that $G = S_m$.
	\begin{proposition}
		Let $a,b \geq 2$, with $m=ab$ and let $G$ be $\mathrm{Sym}(m)$. Let $\Omega$ be the set of partitions of $\{1,2,\dots,m\}$ into $b$ subsets of size $a$, and let $n = |\Omega|$. Then, $b(G) \leq (\log n)/2 + 6$.
	\end{proposition}
	\begin{proof}
		We set $\mathrm{bz}(a,b):=b(G)$ for convenience. We will make frequent use of \cite[Theorem~$1.1$]{MS}, which precisely computes $bz(a,b)$ for any $a,b \in \mathbb{N}$, without restating it here.\\Suppose that $a=2$. Then, $\mathrm{bz}(2,3)=4$ while $\mathrm{bz}(2,b)=3$ for all $b \geq 4$. Thus, $b(G) \leq 4 \leq (\log n)/2+6$ in both cases.	\\
		Suppose now that $b=2$. If $a=4$, $\mathrm{bz}(4,2)=5\leq (\log n)/2+6$ while for $a>4$,  $\mathrm{bz}(a,2) = \lceil \log_2(a+3) \rceil +1$. In this latter case, we have that $n \geq 2^a$, so $\log n \geq a$. Thus,
		\[
			b(G) = \lceil \log(a+3)\rceil+1 \leq  \frac{a}{2}+6 \leq 	\frac{1}{2}\log n+6 .
		\]
		Suppose now that $a,b \geq 3$. If $(a,b) \in \{(3,6),(3,7),(4,7),(7,3)\}$ or $b = a+2$, then $\mathrm{bz}(a,b) \leq 4\leq (\log n)/2+6$. Therefore, the result holds. So suppose that this is not the case. Then, 
		\[
		\mathrm{bz}(a,b) = \lceil \log_b(a+2) \rceil +1.
		\]
		We recall the following inequality, which can be easily obtained by the Stirling formula.
		\[
			x! \geq \left(\frac{x}{3}\right)^x.
		\]
		We have 
		\begin{align*}
			n = \frac{(ab)!}{(a!)^bb!} \geq \frac{\left(\frac{ab}{3}\right)^{ab}}{a^{ab}b^b} = \frac{\left(\frac{b}{3}\right)^{ab}}{b^b} =\frac{b^{(a-1)b}}{3^{ab}} \geq \left(\frac{b}{3^{3/2}}\right)^{(a-1)b}.
		\end{align*}
		Suppose now that $b \geq 11$. In particular, $\log n \geq (a-1)b \log(b/3^{3/2}) \geq 11(a-1)$. Moreover, we have that $\mathrm{bz}(a,b) \leq \lceil \log_{11}(a+2) \rceil +1$. Thus,
		\[
			b(G) \leq \lceil \log_{11}(a+2) \rceil +1\leq a+4 \leq \frac{11}{2}(a-1)+6 \leq \frac{1}{2}\log n+6.
		\]
		Suppose now that $4 \leq b \leq 10$. In this case,
		\[
		n \geq \frac{\left(\frac{b}{3}\right)^{ab}}{b!} \geq \frac{\left(\frac{4}{3}\right)^{ab}}{10!}.
		\]
		In particular, $\log n \geq 4a\log(4/3) -\log 10!$, while $b(G) \leq \lceil \log_b(a+2) \rceil +1 \leq \lceil \log_4(a+2) \rceil +1$. In order to establish the inequality $b(G)\leq (\log n)/2+6$, it suffices to show that 
		\[
		\lceil \log_4(a+2) \rceil +1 \leq \frac{1}{2}\left( 4a\log \frac{4}{3} - \log 10! \right)+6,
		\]
		or, in other words, that the function 
		\[
		f(a) = \frac{1}{2}\left(4a\log \frac{4}{3} - \log 10! \right)+6 - 	\lceil \log_4(a+2) \rceil -1
		\]
		is greater than $0$. Now this is an increasing function and for for $a=10$ its value is greater than $0$, while for $a=9$ it is not. So the initially inequality holds for $a \geq 10$. We can now suppose that $a \leq 9$. 
		\\
		Thus, we have that $4 \leq b \leq 10$ and $ a \leq 9$, with $b \neq a+2$. A computer verification confirmed that the inequality $b(G) \leq (\log n)/2+6$ holds in all these cases.
		\\
		Finally, suppose that $b=3$. Recall the inequality 
		\[
		\binom{m}{k} \geq \left( \frac{m}{k}\right)^k.
		\]
		Using this, we have
		\[
		n = \frac{(3a)!}{(a!)^3 6} \geq \frac{(3a)!}{a!(2a)!6} = \frac{1}{6} \binom{3a}{a} \geq \frac{1}{6} 3^a 
		\]
		so that $\log n \geq a \log3 - \log 6$. Moreover, $\mathrm{bz}(a,3) = \lceil \log_3(a+2) \rceil +1$. Then,
		\[
			b(G) \leq \lceil \log_3(a+2) \rceil +1 \leq \frac{1}{2}(a+4)+1 \leq \frac{1}{2}a\log3-\frac{1}{2}\log 6 +6 \leq \frac{1}{2}\log n+6.
		\]
	\end{proof}
	\subsubsection{The socle of $G$ is a classical group}
	In this section, we take into account all the almost simple groups having socle a classical group $G_0$ acting on $\Omega$, with $(G_0,\Omega)$ not listed in Table \ref{tab:md}.\\
	We will use the following useful results.
	\begin{lemma}$\mathrm{(}$\cite[Lemma~$11$]{MRD}$\mathrm{)}$\label{lemmaQuotCyc}
		Let $G$ be a finite almost simple primitive permutation group acting on $\Omega$ with socle $G_0$, a non-abelian simple classical group, and let $G_0 \trianglelefteq G_1 \trianglelefteq G \leq \mathrm{Sym}(\Omega)$. If $G/G_1$ has a subnormal series of length $s$ with all quotient cyclic, then $b(G) \leq b(G_1)+s$.
	\end{lemma}
	The following theorem summarizes the results obtained in the proof of \cite[Theorem~$3.3$]{HLM}.
	\begin{theorem}\label{basesizeHLM}
		Let $G_0$ be a simple classical group acting primitively on $\Omega$.
		\begin{enumerate}
			\item If $(G_0,k) \neq (\mathrm{P}\Omega^+_{2m},m)$ and $\Omega= S_k$, then $b(G_0) \leq d/k+10$.
			\item If $\Omega= N_k$, then $b(G_0) \leq d/k+11$.
			\item If $G_0 = \mathrm{PSL}_d(q)$ and $\Omega = P_k$, then $b(G_0) \leq d/k+5$.
		\end{enumerate}
	\end{theorem}
	\begin{remark}
		In this and in the following sections, we will use \cite[Table~$4.1.2$]{BGiu} to identify the degrees of all the subspace actions that will appear.
	\end{remark}
	\begin{proposition}
		Let $G$ be an almost simple group having socle $G_0=\mathrm{PSL}_d(q)$, with $d\geq 2$, acting on $P_k$, for $1 \leq k \leq d/2$, and let $n=|P_k|$. Suppose moreover that if $k=1$ then $q>3$. Then, $b(G) \leq (\log n)/2 + 6$.
	\end{proposition}
	\begin{proof}
		By Theorem~\ref{basesizeHLM}, we have
		\[
			b(G_0) \leq \frac{d}{k}+5.
		\]
		The group $\mathrm{Out}(G_0)$ has a normal series with all cyclic quotients of length at most three, $G/G_0$	has such a series with length at most two if $k = d/2$, or if $d = 2$, or if $q$ is prime; and is cyclic if more than one of these conditions holds. Thus, by Lemma~\ref{lemmaQuotCyc}, $b(G) \leq b(G_0) + 3$, so that 
		\[
			b(G) \leq \frac{d}{k}+8.
		\]
		Firstly, suppose that $k=1$, so that $q>3$. It is easy to see that $G_0$ admits a base of cardinality $d+1$. In particular, by Lemma~\ref{lemmaQuotCyc}, $b(G) \leq b(G_0)+1 \leq d+4$. Moreover, since $n$ is an increasing function of $q$ with $q\geq 4$, and $d>1$, we have that
		\[
			\log n \geq \log \left( \frac{4^d-1}{3}\right) \geq  \log \left(\frac{4^d-4^{d-1}}{3}\right) = 2d-2.
		\]
		Thus,
		\[
			b(G) \leq d+4 \leq \frac{1}{2}(2d-2)+6 \leq \frac{1}{2}\log n +6.
		\]
		Suppose now that $k=2$, so that $d \geq 4$. Firstly suppose that $d=4$. If $q=2$ or $q=3$, a GAP \cite{GAP} computation shows that our statement holds. Suppose $q>3$. In this case, by \cite[Lemma $4$]{MRD}, we have $b(G) \leq 8$, and $8 \leq (\log n)/2+6$ if and only if $n \geq 16$. But $n=(q^2+1)(q^2+q+1)$ and for $q\geq 4$ we have $n \geq 16$. Finally, suppose that $d >4$. Again by \cite[Lemma $4$]{MRD}, we have 
		\[
		b(G) \leq \left \lceil \frac{d}{2} \right\rceil + 5,
		\]
		and $\log n \geq 2d-4$, by \cite[Proposition $6$]{MRD}. Thus,
		\[
			b(G) \leq \left\lceil \frac{d}{2} \right\rceil + 5 \leq \frac{d}{2}+6 \leq d+4 =\frac{1}{2}(2d-4)+6\leq \frac{1}{2}\log n+6.
		\]
		Suppose now that $k \geq 3$, so that $d \geq 6$. We have seen that $b(G) \leq d/k+8 \leq d/3+8$. Moreover, $\log n \geq d+3$ (see \cite[Proposition $6$]{MRD}). Thus,
		\[
			b(G) \leq \frac{d}{3}+8 \leq \frac{1}{2}(d+3)+6 \leq \frac{1}{2}\log n+6.
		\]
	\end{proof}
	\begin{proposition}\label{PSUtotsing}
		Let $G$ be an almost simple group having socle $G_0=\mathrm{PSU}_d(q)$, with $d\geq 3$, acting on $S_k$, for $1 \leq k \leq d/2$, and let $n=|S_k|$. Then, $b(G) \leq (\log n)/2+6$.
	\end{proposition} 
	\begin{proof}
	By Theorem~\ref{basesizeHLM},
		\begin{equation}\label{dk-ineq}
			b(G_0) \leq \frac{d}{k}+10.
		\end{equation}
		Moreover, $\mathrm{Aut}(G_0)/\mathrm{PGU}_d(q)$ is cyclic, while $\mathrm{Out}(G_0)$ has a normal series with two cyclic quotients. Thus, by Lemma~\ref{lemmaQuotCyc}, $b(G) \leq b(G_0)+2$, and $b(G) \leq b(\mathrm{PGU}_d(q)) +1$.\\
		Firstly, suppose that $k=1$. By \cite[Lemma~$3$]{MRD}, we have $b(\mathrm{PGU}(d,q)) \leq d$, so that $b(G) \leq d+1$. From \cite[Table~$4.1.2$]{BGiu}, we see that
		\[
		n = \frac{(q^{d-1}-(-1)^{d-1})(q^d-(-1)^d)}{q^2-1}.
		\]
		Suppose first that $d$ is even. Hence, we have
		\begin{align*}
			n &= \frac{(q^{d-1}+1)(q^d-1)}{q^2-1} \geq \frac{q^{d-1}}{q^2}(q^d-q^{d-1}) =q^{d-3}q^{d-1}(q-1) \geq q^{2d-4}.
		\end{align*}
		Suppose now that $d$ is odd, so that
		\begin{align*}
			n = \frac{(q^{d-1}-1)(q^d+1)}{q^2-1} \geq \frac{q^d(q^{d-1}-q^{d-2})}{q^2} = q^{2d-4}(q-1) \geq q^{2d-4}.
		\end{align*}
		In both cases, $\log n \geq 2d-4$. Thus,
		\[
			b(G) \leq d+1 \leq \frac{1}{2}(2d-4)+6 \leq \frac{1}{2}\log n+6.
		\]		
		Suppose now that $k=2$. By \cite[Lemma $5$]{MRD}, we see that
		if $d \leq 6$, then $b(G) \leq 6 \leq (\log n)/2+6$. if $d\geq 7$, then $b(G) \leq \lceil d/2 \rceil +1$. Also in this case $\log n \geq 2d-4$ (see \cite[Proposition $5$]{MRD}). Thus, $b(G) \leq d/2+2 \leq (2d-4)/2+6 \leq (\log n)/2 +6.$\\ 
		Finally, suppose that $k \geq 3$, so that $d \geq 6$. We have $b(G) \leq d/3 + 12$. Moreover, by \cite[Proposition $5$]{MRD}), we see that $\log n \geq 3d-9$. Suppose now that $d\geq 9$. We have
		\[
			b(G) \leq \frac{d}{3}+12 \leq \frac{1}{2}(3d-9)+6 \leq \frac{1}{2}\log n +6.
		\]
		Suppose now that $d\leq 9$. Then, there are just few possibilities: $(d,k) = (6,3),(7,3),(8,3),(8,4)$. \\
		Suppose that $(d,k) = (6,3)$. In this case,
		\[
		n = (q+1)(q^3+1)(q^5+1),
		\]
		while $b(G) \leq 14$. It is easy to see that if $q \geq 4$, then $14 \leq (\log n)/2 + 6$.
		If $q=2$ or $q=3$, we have verified with $\mathrm{GAP}$ \cite{GAP} that $G_0$ has a base of cardinality $6$ and $5$ respectively\footnote{For $q=2$, we used the $\mathrm{GAP}$ library of primitive groups. For $q=3$, we used the FinInG package \cite{FINING}}. Thus, $b(G) \leq 8$ in both cases. Now if $q=2$ then $n=891$ and if $q=3$ then $n=27328$, and thus $\log n \geq 4$.\\ 
		Suppose that $(d,k) = (7,3)$. Thus,
		\[
		n = (q^3+1)(q^5+1)(q^7+1),
		\]
			and $b(G) \leq 7/3+12$. For $q \geq 3$, we have that $7/3+12 \leq (\log n)/2+6$, since $n$ is an increasing function of $q$ and the inequality $b(G) \leq (\log n)/2+6$ holds for $q=3$. For $q=2$, we have verified with $\mathrm{GAP}$ that $G_0$ has a base of cardinality $4$, hence $b(G) \leq 6 \leq 6+(\log n)/2$.\\
		Suppose now that $(d,k)= (8,3)$. In this case,
		\[
			n = (q^2+1)(q^3+1)(q^4+1)(q^5+1)(q^7+1),
		\]
		while $b(G) \leq 8/3 +12$. Since $n$ is an increasing function in $q$ and the value of $(\log n)/2+6$ for $q=2$ is greater then $8/3 +12$, we are done.\\
		Suppose finally that $(d,k) = (8,4)$. In this case,
		\[
		n = (q+1)(q^3+1)(q^5+1)(q^7+1),
		\]
		while $b(G) \leq 14$. As before, for $q=2$, we have that $(\log n)/2+6 \geq 14$, and so we are done since $n$ is in increasing function of $q$.
	\end{proof}
	
	\begin{proposition}
		Let $G$ be an almost simple group having socle $G_0=\mathrm{PSU}_d(q)$, with $d\geq 3$, acting on $N_k$ for $1 \leq k < d/2$, and let $n = |N_k|$. Then, $b(G) \leq (\log n)/2+6$.
	\end{proposition} 
	\begin{proof}
		As we already seen in Proposition~\ref{PSUtotsing}, $b(G) \leq b(G_0)+2$ and $b(G) \leq b(\mathrm{PGU}_d(q))+1$. Combining these with Theorem~\ref{basesizeHLM},  we get $b(G)\leq d/k+13$.\\
		Suppose firstly that $k=1$. In this case, \cite[Lemma $6$]{MRD}  shows that $b(\mathrm{PGU}_d(q)) \leq d$, so that $b(G) \leq d+1$. Moreover, we have that
		\[
			n = \frac{q^{d-1}(q^d-(-1)^d)}{q+1},
		\]
		and it is easy to see, by dividing cases where $d$ is odd or even, that
		\[
		n \geq
		\begin{cases}
			q^{2d-4} \text{ if $d$ is even,} \\
			q^{2d-3} \text{ if $d$ is odd,}
		\end{cases}
		\]
		so that $\log n \geq 2d-4$ in both cases. In conclusion, $b(G) \leq d+1 \leq (2d-4)/2+6 \leq (\log n)/2+6$.\\
		Suppose that $k=2$. Then, \cite[Lemma $9$]{MRD} shows that $b(G) \leq \lceil d/2 \rceil +1$ if $d\neq 6$. If $d=6$, it shows that $b(G) \leq 5\leq (\log n)/2+6$, so we can suppose that $d \neq 6$. By dividing the cases where $d$ is odd or even, we have 
		\[
		n = \frac{q^{2(d-2)}(q^{d-1}-(-1)^{d-1})(q^d-(-1)^d)}{(q+1)(q^2-1)} \geq q^{4d-10}.
		\]
		Thus, $b(G) \leq d/2+2 \leq (4d-10)/2+6 \leq (\log n)/2+6$.	\\
		Finally, suppose that  $k \geq 3$, so that $d > 6$. We have $b(G) \leq d/3+13$. Moreover, by  \cite[Proposition $5$]{MRD}), we see that $\log n \geq d+11$. If $d\geq 9$, we have
		\[
			b(G) \leq \frac{d}{3}+13 \leq \frac{1}{2}(d+11)+6 \leq \frac{1}{2}\log n+6.
		\]
		Thus, we have to argue separately for the cases $(d,k)=(7,3)$ and $(d,k)=(8,3)$. \\
		Suppose first that $(d,k) = (7,3)$. In this case,
		\[
		n = \frac{q^{12}(q^3-1)(q^5+1)(q^7+1)}{(q+1)(q^2-1)},
		\]
		while $b(G) \leq 46/3$. Notice that $n$ is an increasing function on $q$, and for $q=2$ we have 
		\[
		\frac{1}{2}\log n+6 \geq \frac{46}{3},
		\]
		and so we are done. The case $(d,k) = (8,3)$ is done in the same way.
	\end{proof}

	\begin{proposition}\label{propOrtSing}
		Let $G$ be an almost simple group having socle $G_0=\mathrm{P}\Omega_d^\varepsilon(q)$, with $d \geq 8$ even, acting on $S_k$, and let $n=|S_k|$. Suppose moreover that $q>3$ when $k=1$. Then $b(G) \leq (\log n)/2+6$.
	\end{proposition}
	\begin{proof}
		If $G_0 = \mathrm{P}\Omega_8^+(p^f)$ with $p$ odd, then $\mathrm{Out}(G_0) = \mathrm{Sym}(4) \times C_f$, while if $p=2$ then $\mathrm{Out}(G_0)=\mathrm{Sym}(3)\times C_f$. If $G_0 \neq  \mathrm{P}\Omega_8^+(p^f)$ or if the action is not on $S_2$, then $G/G_0$ has a normal series with cyclic quotients of length at most three and  $G/\mathrm{PGO}^\varepsilon_d(q)$ has a normal series with cyclic quotients of length at most $2$. Thus, $b(G) \leq b(G_0)+3$ and $b(G) \leq b(\mathrm{PGO}_d^\varepsilon(q))+2$ by Lemma~\ref{lemmaQuotCyc}. By Theorem~\ref{basesizeHLM}, if $3\leq k \leq d/2-1$, then $b(G_0) \leq d/k+10$, and so 
		\[
			b(G) \leq \frac{d}{k}+13.
		\]
		First of all, suppose that $k=1$, so that $q>3$. By \cite[Lemma $3$]{MRD}, $b(\mathrm{PGO}_d^\varepsilon(q) ) \leq d-1$, and so $b(G) \leq d+1$. In this case, using the fact that $q\geq 4$,
		\[
		n = \frac{(q^{\frac{d}{2}}\mp 1)(q^{\frac{d}{2}-1}\pm 1)}{q-1} \geq \frac{(4^{\frac{d}{2}}\mp 1)(4^{\frac{d}{2}-1}\pm 1)}{3}\geq 4^{d-2},
		\]
		so that $\log n \geq 2(d-2)$. Thus, $b(G) \leq d+1 \leq d+4 \leq (\log n)/2+6$.\\
		Suppose now that $k=2$. By \cite[Lemma~$5$]{MRD}, $b(\mathrm{PGO}_d^\varepsilon(q)) \leq \lceil d/2 \rceil$, and so, $b(G) \leq d/2+5$, while $\log n  \geq d-1$ (see \cite[Proposition~$2$]{MRD}). Thus $b(G) \leq d/2+5 \leq (d-1)/2+6\leq (\log n)/2+6$.
		\\
		Suppose now that $3 \leq k \leq d/2-1$. We firstly take into account the case $(d,k)=(8,3)$. If $G_0 = \mathrm{P}\Omega_8^+(q)$, we have that
		\[
			n = \frac{(q^4-1)(q+1)(q^2+1)(q^3+1)}{(q-1)},
		\]
		while $b(G) \leq 8/3+15\leq 18$ if $q$ is odd and $b(G) \leq 8/3+14\leq 17$ if $q$ is even. For $q=7$ we have $(\log n)/2+6 \geq 18$, and for $q=8$ we have $(\log n)/2+6 \geq 17$. Since $n$ is an increasing function of $q$, we have $b(G) \leq (\log n)/2+6$ for all values of $q\geq 7$.	We have thus to handle with the case $G_0 = \mathrm{P}\Omega_8^+(q)$ acting on $S_3$ with $q\in \{2,3,4,5\}$. Using the $\mathrm{GAP}$ package $\mathrm{FinInG}$ \cite{FINING}, we can compute bases for $G_0$ in these cases. In particular, we have found that $G_0$ has a base of cardinality $4$ for all these values of $q$, and so $b(G) \leq 9$. However, $(\log n)/2 +6\geq 9$ in all the four cases.\\
		If $G_0 = \mathrm{P}\Omega_8^-(q)$, then $b(G) \leq b(G_0)+3 \leq 8/3+13\leq 16$. In this case, 
		\[
			n = (q^2+1)(q^3+1)(q^4+1),
		\] 
		and $(\log n)/2 + 6 \geq 16$ for $q=5$, and hence for all $q\geq 5$. For $q\in\{2,3,4\}$, we argue as in the hyperbolic case.\\
		Suppose now $d \geq 10$. We have $b(G) \leq d/3+13$, and by \cite[Proposition~$2$]{MRD}, we see that $\log n \geq d+5$. If $d\geq 27$, we have $b(G) \leq d/3+13 \leq (d+5)/2+6 \leq (\log n)/2+6$.\\
		Thus, we can now suppose that $10 \leq d \leq 26$. If $q=3$, we use the exact value of $n$ to see that, for each possible choice of $(d,k)$, we have
		\[
		\frac{d}{3}+13 \leq \frac{1}{2}\log n+6,
		\]
		in both the cases $\varepsilon = \pm$. Since $n$ is always an increasing function of $q$, the inequality holds for all $q \geq 3$.\\
		Suppose then that $q=2$. Using the exact value of $n$, we see that the inequality $d/3+13 \leq (\log n)/2+6$ holds except for $(d,k,\varepsilon) = (10,3,\pm)$ or $(d,k,\varepsilon)=(10,4,\pm)$ and $(d,k,\varepsilon) = (12,5,-)$. However, in these cases, we can compute a base using the $\mathrm{GAP}$ package $\mathrm{FinInG}$ \cite{FINING}. 
		In particular, if $(d,k,\varepsilon)=(10,3,\pm)$, $G_0$ admits a base of cardinality $4$, so that $b(G) \leq 7 \leq (\log n)/2+6$.\\
		If $(d,k,\varepsilon) = (10,4,\pm)$, $G_0$ admits a base of cardinality $5$ and $4$ respectively, so that $b(G) \leq 8 \leq (\log n)/2+6$.\\
		If $(d,k,\varepsilon) = (12,5,-)$, $G_0$ admits a base of cardinality $4$, so that $b(G) \leq 6 \leq (\log n)/2 +6$.
		\\ 
		Finally, suppose that $k = d/2$, so that $\epsilon = +$. In this case $\log n \geq d(d+2)/8$ while $b\leq 12$ If $d\geq 10$, we have $b(G) \leq 12 \leq d(d+2)/16+6 \leq (\log n)/2+6$.\\
		Assume that $d = 8$. In this case,
		\[
			n = (q+1)(q^2+1)(q^3+1)(q^4+1).
		\]
		It is shown in \cite{HLM} that $b(G) \leq 10$ when $q=2$ and $b(G) \leq 12$ otherwise. If $q=2$ we he have that $(\log n)/2+6 \geq 11$. If $q\geq3$, we have $(\log n)/2+6 \geq 16$, and this concludes the proof.
	\end{proof}
	\begin{proposition}
		Let $G$ be an almost simple group having socle $G_0=\mathrm{P}\Omega_d^\varepsilon(q)$, with $d \geq 8$ even, acting on $N_k^\epsilon$ with $\epsilon=+,-,\circ$, and let $n=|N_k^\epsilon|$. Suppose moreover that $q>3$ when $k=1$. Then $b(G) \leq (\log n)/2+6$.
	\end{proposition}
	\begin{proof}
		The stabilizer of an element of $N_k^\epsilon$ also stabilizes a non-degenerate $(d-k)$-space, of the opposite sign if $\varepsilon = -$ and $k$ is even, and of the same sign otherwise. Thus, we may assume $k \leq d/2$.\\
		Suppose firstly that $k$ is even, so that $2 \leq k \leq d/2$, and if $k=d/2$ then $\varepsilon = -$ (otherwise the action is not primitive). If $k=2$, then $\log n \geq d+2$ and $b \leq d/2+2$ (see Lemma~$9$ and Lemma~$10$ of \cite{MRD}). Thus, $b(G) \leq (\log n)/2+6$. \\
		Suppose now that $k >2$. By Theorem~\ref{basesizeHLM}, and using the information about $\mathrm{Out}(G_0)$ given in Proposition~\ref{propOrtSing}, we have $b(G) \leq d/4+14$. In \cite[Proposition~$2$]{MRD} it is proved that $\log n \geq 4d-17$. If $d \geq 10$, we have that $b(G) \leq d/4+14 \leq (4d-17)/2+6 \leq (\log n)/2+6$. So we must deal separately with the case $(d,k) = (8,4)$. In this case, $\varepsilon = -$. Suppose firstly that $\epsilon = +$. Here, we have
		\[
			n = \frac{q^8(q^6-1)(q^4+1)}{2(q^2-1)}.
		\] 
		If $q=2$, then $\mathrm{Out}(G_0)$ is cyclic, hence $b(G) \leq 14$ by Lemma~\ref{lemmaQuotCyc}, while $b(G) \leq 16$ if $q>2$. We have that $16 \leq (\log n)/2+6$ if and only if $\log n \geq 20$. However, $\log n \geq 20$ for $q>2$, while $\log n > 14$ if $q=2$, and so we are done. The case $\epsilon = -$ is similar.

		Suppose now that $k$ is odd, so that $1 \leq k < d/2$. If $k=1$, then $b(G) \leq d+1$ (see \cite[Lemma $7$]{MRD}), while
		\[
			n = \frac{q^{\frac{d-2}{2}}(q^{\frac{d}{2}}\mp 1)}{(2,q-1)} \geq \frac{4^{\frac{d-2}{2}}(4^{\frac{d}{2}}\mp 1)}{2} = 2^{d-3}(2^d\mp 1),
		\]
		since $n$ is an increasing function of $q$ and $q>3$ by hypothesis.  In particular, $\log n \geq 2d-4$, and thus we have
		\[
			b(G) \leq d+1 \leq 	\frac{1}{2}(2d-4)+6 \leq \frac{1}{2}\log n+6.	
		\] 
		Suppose finally that $k>1$, so that $q$ is odd. Then $b(G) \leq d/3+14$, while \cite[Proposition~$2$]{MRD} shows that $\log n \geq (3d-10)\log(q/2)$. 
		If $q>3$, then $\log n \geq 3d-10$. If $d\geq 12$, we have
		\[
			b(G) \leq \frac{d}{3}+14 \leq \frac{1}{2}(2d-10)+6 \leq \frac{1}{2}\log n +6,
		\] 
		So the statement is proved for $d\geq 12$ and $q>3$. We now deal with cases $(d,k) = (8,3)$ and $(d,k)=(10,3)$ under the hypothesis $q>3$.\\
		Suppose that $(d,k)=(8,3)$. Here, 
		\[
			n = \frac{1}{2}q^7(q^4-\varepsilon)\frac{q^6-1}{q^2-1},
		\]
		while $b(G) \leq 50/3$. But $50/3 \leq (\log n)/2+6$ for $q=5$ in both cases, and since $n$ is an increasing function of $q$, we are done. The case $(d,k)=(10,3)$ is done in a similar way.\\
		Suppose then that $q=3$. If $d\geq 22$, we have 
		\[
			b(G) \leq \frac{d}{3}+14 \leq \frac{1}{2}(2d-4)\log\frac{3}{2}+6 \leq \frac{1}{2}\log n+6.
		\]
		If $8 \leq d \leq 20$, we checked the validity of the inequality using the exact value of $n$ in all the possible cases.
	\end{proof}

	\begin{proposition}
		Let $G$ be an almost simple group having socle $G_0=\mathrm{P}\Omega_d(q)$, with $d \geq 7$, acting on  $S_k$, with $1 \leq k \leq (d-1)/2$, and let $n=|S_k|$. Suppose moreover that $q>3$ when $k=1$. Then $b(G) \leq (\log n)/2+6$.
	\end{proposition} 
	\begin{proof}
		By Theorem~\ref{basesizeHLM}, $b(G_0) \leq d/k+10$. Moreover, since $\mathrm{Out}(G_0)$ has a normal series with at most two cyclic quotients, $b(G) \leq b(G_0)+2$  by Lemma~\ref{lemmaQuotCyc}, so that 
		\[
		b(G) \leq d/k+12.
		\]
		Suppose that $k=1$, so that $q > 3$. By \cite[Lemma $3$]{MRD}, $b(G) \leq d+1$, while $n=(q^{d-1}-1)/(q-1)\geq (4^{d-1}-1)/3 \geq 4^{d-2}$, so that $\log n \geq 2(d-2)$. Thus, $b(G) \leq d+1 \leq d+4 \leq (\log n)/2+6$. \\
		Suppose now that $k=2$, so that $d \geq 5$. Then \cite[Lemma $5$]{MRD} gives $b(G) \leq \lceil d/2 \rceil +1\leq d/2+2$, while
		\[
			n = \frac{\prod_{(d-3)/2}^{(d-1)/2}(q^{2i}-1)}{\prod_{i=1}^{2}(q^i-1)} = \frac{(q^{d-3}-1)(q^{d-1}-1)}{(q-1)^2(q+1)}\geq q^{2d-7}.
		\]
	 	Thus, $b(G) \leq d/2+2 \leq (2d-7)/2+6 \leq (\log n)/2+6$.\\
		Suppose now that $3 \leq k \leq (d-3)/2$, so that $d \geq 9$. We have $b(G) \leq d/3+12$. Moreover, from \cite[Proposition~$3$]{MRD}, we see that $\log n \geq d+3$. Suppose that $d\geq 27$. Then,
		\[
			b(G) \leq \frac{d}{3}+12 \leq \frac{1}{2}(d+3)+6 \leq \frac{1}{2}\log n+6.
		\]
		For $9 \leq d \leq 25$, we have verified the inequality $d/k+12 \leq (\log n)/2+6$ using the exact value of $n$. In particular, we have verified that for $q = 3$, we have 
		\[
		\frac{d}{k}+12 \leq \frac{1}{2}\log n+6.
		\] 
		Using the fact that $n$ is an increasing function of $q$, we have the desired inequality for all $q\geq 3$. For $q=2$, we again use the exact value of $n$ to see that $d/k+12 \leq (\log n)/2+6$ holds except for $(d,k)=(9,3)$. However, in this case, we have verified with $\mathrm{GAP}$ that $G_0$ has a base of cardinality $10$, thus $b(G) \leq b(G_0)+2 \leq 12$, and $12 \leq (\log n)/2+6$, so we are done. \\
		Finally, suppose that $k = (d-1)/2$. Then $b(G) \leq d/k+12 = 2d/(d-1)+12$. By \cite[Proposition~$3$]{MRD}, we have $\log n \geq (d^2-1)/8$. If $d\geq 11$, then $b(G) \leq 2d/(d-1)+12 \leq (d^2-1)/16+6 \leq (\log n)/2+6$.\\
		We must deal with the case $(d,k)=(9,4)$. In this case,
		\[
			n = (q+1)(q^2+1)(q^3+1)(q^4+1),
		\]
		and $b(G) \leq 69/5$. Now, $69/5 \leq (\log n)/2+6$ for all $q\geq 3$. For $q=2$, we have verified with $\mathrm{GAP}$ that $G_0$ has a base of cardinality $6$, and as before we are done.
	\end{proof}
	
	\begin{proposition}
		Let $G$ be an almost simple group having socle $G_0=\mathrm{P}\Omega_d(q)$, with $d \geq 7$, acting on $N_k^\epsilon$, with $1 \leq k \leq (d-1)$ even, and let $n=|N_k^\epsilon|$. Suppose moreover that $q>3$ if $d=k-1$. Then $b(G) \leq (\log n)/2+6$, where $n=|N_k^\epsilon|$.
	\end{proposition}
	
	\begin{proof}
		By Lemma~\ref{lemmaQuotCyc} and Theorem~\ref{basesizeHLM},
		\[
			b(G) \leq \frac{d}{k}+13.
		\]
		Suppose firstly that $k=2$. In this case $\log n \geq d$ (\cite[Proposition~$3$]{MRD}), while $b(G) \leq \lceil d/2 \rceil +1 \leq d/2+2$ (\cite[Lemma~$10$]{MRD}). We have moreover that 
		\[
			n = \frac{q^{d-2}(q^{d-1})}{2(q\mp 1)}\geq q^{2d-6},
		\]
		and so $\log n \geq 2d-6$. Thus, $b(G) \leq d/2+2 \leq d+3 \leq (\log n)/2+6$.\\
		Suppose now that $4 \leq k \leq d-3$, so that $\log n \geq (3d-9)\log q-2$ (see \cite[Proposition~$3$]{MRD}). If $q\geq 3$, and $d>7$, we have
		\[
			b(G) \leq \frac{d}{4}+13 \leq \frac{1}{2}(3d-9)\log 3 +5 \leq \frac{1}{2}\log n +6.
		\] 
		If $d=7$ (and then $k=4$), a direct computation using the exact value of $n$ shows that, for $q=3$, we have $7/4+13 \leq (\log n)/2+6$, and since $n$ is an increasing function of $q$, we are done.\\		
		Suppose now that $q=2$ and that $d\geq 9$. Thus, $b(G) \leq d/4+13 \leq (3d-9)/2+6 \leq (\log n)/2+6$. If $(d,k,q) = (9,4,2)$, we have verified with $\mathrm{GAP}$ that, in both cases, $b(G) \leq 6\leq (\log n)/2+6$.\\
		Finally suppose that $k=d-1$, so that $q>3$. Then, $b(G) \leq d$: this follows by noting that the action is equivalent to the one of $G$ on $N_1^\epsilon$, and from \cite[Lemma~$7$]{MRD}. By \cite[Proposition~$3$]{MRD} $n \geq (1/4)q^{d-1} \geq 4^{d-2}$, so that $\log n \geq 2(d-2)$. Thus, $b(G) \leq d \leq d+4 \leq (\log n)/2+6$, and this concludes the proof.		
	\end{proof}
	\begin{proposition}
		Let $G$ be an almost simple group having socle $G_0 = \mathrm{PSp}_d(q)$, with $d\geq 4$, acting on $S_k$, with $1 \leq k \leq d/2$, and let $n=|S_k|$. Suppose moreover that $q>3$ when $k=1$. Then, $b(G) \leq (\log n)/2+6$.
	\end{proposition}
	\begin{proof}
		$\mathrm{Out}(G_0)$ has a normal series with all quotient cyclic of length at most two. Thus, by Lemma~\ref{lemmaQuotCyc} and Theorem~\ref{basesizeHLM},
		\[
			b(G) \leq \frac{d}{k}+12.
		\]
		If $k=1$, then $m=(q^d-1)(q-1)$ while $b(G) \leq d+2$, and the result follows with the same computation as in the linear case.\\
		Suppose that $k=2$. By \cite[Lemma~$5$]{MRD}, if $d = 4$, then $b(G) \leq 5$ and if $d > 4$ then $b(G) \leq d$. In each case, a routine computation shows that $b(G) \leq (\log n)/2+6$, since $\log n \geq 2d-5$ (see \cite[Proposition $4$]{MRD}). \\
		Suppose now that $3 \leq k \leq d/2-1$. Then $b(G) \leq d/3+12$. Moreover, by \cite[Proposition~$4$]{MRD}, $\log n \geq d+5$. If $d\geq 21$, we have $b(G) \leq d/3+12 \leq (d+5)/2+6 \leq (\log n)/2+6$. For $d=4,\dots,21$ and all the possible values of $k$, we have checked the validity of the inequality using the exact value of $n$. In particular, we have checked it for $q=2$ and then we conclude using the fact that $n$ is an increasing function of $q$.\\
		Finally, suppose that $k=d/2$, so that $b(G) \leq 14$. In this case,
		\[
			n=\prod_{i=1}^{d/2}(q^i+1).
		\]
		If $d=10$ and $q=2$, then
		\[
		\frac{1}{2}\log n + 6 \geq 14
		\]
		Since $n$ is an increasing function of both $d$ and $q$, we have that $b(G) \leq 14 \leq (\log n)/2+6$ for all $d\geq 10$ and $q$.\\
		Suppose now that $d\leq8$.
		If $d=4$, \cite[Lemma~$5$]{MRD} shows that $b(G_0) \leq 4$, so that $b(G) \leq 6 \leq (\log n)/2+6$.\\
		Suppose that $d = 6$. Thus, $n = (q+1)(q^2+1)(q^3+1)$, and 
		\[
		14 \leq \frac{1}{2}\log n+6
		\]
		for $q \geq 7$. For $q\in\{2,3,4,5\}$, we have verified with $\mathrm{GAP}$ the inequality.\\
		The case $d=8$ is done in the same way.
	\end{proof}
	
	\begin{proposition}
		Let $G$ be an almost simple group having socle $G_0 = \mathrm{PSp}_d(q)$ with $d\geq 4$, acting on $N_k$ , with $1 \leq k \leq d/2-1$ even, and let $n=|N_k|$. Then, $b(G) \leq (\log n)/2+6$.
	\end{proposition}
	\begin{proof}
		As in the previous proposition, using Lemma~\ref{lemmaQuotCyc} and Theorem~\ref{basesizeHLM}, we have
		\[
			b(G) \leq \frac{d}{k}+12.
		\]
		Suppose that $k=2$. Then by \cite[Proposition~$4$]{MRD} $\log n \geq 2d-4$ while, by \cite[Lemma~$9$]{MRD}, $b(G) \leq d$. Using these information we get $b(G) \leq d \leq (\log n)/2+6$. If $k>2$, then \cite[Proposition~$4$]{MRD} shows that $\log n \geq d+13$. Thus, $b(G) \leq d/3+12 \leq (d+13)/2+6 \leq (\log n)/2+6$.
	\end{proof}
	Finally, we take into account groups of type~$2.2$ of Definition \ref{defintionsubspacesubgrp}. 	
	\begin{proposition}
		Let $G$ be an almost simple group with socle $G_0 = \mathrm{PSp}_d(q)$, with $d\geq 4$. Let $M = N_G(\mathrm{GO}_d^\epsilon(q))$, with $q\geq 4$ even, and let $\Omega=[G:M]$ be the right coset space, with $n = |\Omega|$. Then, $b(G) \leq (\log n)/2+6$.
	\end{proposition}
	\begin{proof}
		In \cite{HLM}, it is proved that $b(G_0) \leq d+1$. Moreover, since $\mathrm{Out}(G_0)$ is cyclic, by Lemma~\ref{lemmaQuotCyc}, $b(G) \leq d+2$. Using \cite{BGiu}, we see that $n =q^{d/2}(q^{d/2}+\epsilon)/2$. Since $q\geq 4$, $n \geq 2^d(2^d+\epsilon)/2\geq 2^{2d-2}$, so that $\log n \geq 2d-2$. Thus, $b(G) \leq d+2 \leq d+5 \leq (\log n)/2+6$.
	\end{proof}
	We now take into account all the remaining subspace actions. By Definition~\ref{defintionsubspacesubgrp}, consulting \cite{bhr} we see that the remaining cases are the following.
	\begin{enumerate}
		\item\label{HiHo1} $G_0 = \mathrm{PSL}_d(q)$ and $G$ is acting on pair of subspaces,
		\item\label{HiHo2} $G_0=\mathrm{Sp}_4(2^a)$, $G$ contains a graph automorphism and $G_0$ is acting on the cosets of a local subgroup $H$ of type $[q^4]:C_{q-1}^2$, see~\cite[Table~8.14]{bhr},
		\item\label{HiHo3} $G_0=\mathrm{P}\Omega_8^+(q)$, $G$ contains a triality graph automorphism and $G_0$ is acting on the cosets of a local subgroup $H$ of type $$[q^{11}]:\left[\frac{q-1}{d}\right]^2\cdot\frac{1}{d}\mathrm{GL}_2(q)\cdot d^2,$$where $d=\gcd(2,q-1)$, see~\cite[Table~8.50]{bhr},
		\item\label{HiHo4} $G_0=\mathrm{P}\Omega_8^+(q)$, $G$ contains a triality graph automorphism and $G_0$ is acting on the cosets of a subgroup $H$ isomorphic to $G_2(q),$ see~\cite[Table~8.50]{bhr},
	\end{enumerate}

	We begin with the actions of linear groups acting on pair of subsets. To do this, we need to describe these actions, and firstly we set some notation.\\
	Recall that, for $d \geq 3$, we have
	\[
		\mathrm{Aut}(\mathrm{PSL}_d(q)) = \mathrm{P}\Gamma\mathrm{L}_d(q):\langle \iota \rangle,
	\]
	where $\iota$ is the graph automorphism of $G$, defined by
	\begin{align*}
		\iota : \mathrm{PSL}_d(q) &\to \mathrm{PSL}_d(q) \\
				x &\mapsto (x^{-1})^T,
	\end{align*}
	where $g^T$ is the transpose of $g$. \\
	Now if $\mathrm{PSL}_d(q) \trianglelefteq G \nleq \mathrm{P}\Gamma \mathrm{L}_d(q)$, then $G$ acts primitively in the following domain. For $1\leq k < d/2$, we let
	\begin{align*}
		\Omega_k^1 &= \{\{U,W\}\, : \, \dim U = k, U \oplus W = V\},\\
		\Omega_k^2 &= \{\{U,W\}\, : \, \dim U = k, \dim W = n-k, U \subseteq W\},
	\end{align*}
	where $\iota$ interchanges the two subspaces of the pair ${U,W} \in \Omega_k^i$, with $i=1,2$. 
	\begin{proposition}
		Let $\mathrm{PSL}_d(q) \trianglelefteq G \nleq \mathrm{P}\Gamma \mathrm{L}_d(q)$ acts primitively on $\Omega_k^1$ or $\Omega_k^2$ for $1 \leq k < d/2$ and let $n=|\Omega_k^i|$, for $i=1,2$. Then, $b(G) \leq (\log n)/2+6$.
	\end{proposition}
	\begin{proof}
		We start with $\Omega_k^1$.\\
		As noted in \cite{HLM}, we have that $b(G) \leq b_k(G)$, where with $b_k(G)$ we denote the base size of $G$ acting on the set $P_k$ of $k$-dimensional subspaces of $V$. Thus, using Theorem~\ref{basesizeHLM} and Lemma~\ref{lemmaQuotCyc}, we have that $b(G) \leq d/k+8$. Using \cite{BGiu}, we see that
		\[
			n =  q^{k(d-k)} \prod_{i=1}^{k}\frac{q^{d-k+i}-1}{q^i-1}.
		\]
		Suppose first that $k=1$. In this case, $b(G) \leq b_1(G) \leq d+4$, while
		\[
			n = \frac{q^{d-1}(q^d-1)}{q-1} \geq q^{2d-2}.
		\]
		Thus, $b(G) \leq d+4 \leq d+5 \leq (\log n)/2+6$. \\
		Suppose now that $k = 2$, so that $d>4$. By \cite[Lemma~$4$]{MRD}, $b(G) \leq b_k(G) \leq \lceil d/2 \rceil +3 \leq d/2+4$. Moreover,
		\[
			n = \frac{q^{2(d-2)}(q^{d-1}-1)(q^d-1)}{(q-1)(q^2-1)}\geq q^{4d-9}.
		\]
		Now we have $b(G) \leq d/2+4 \leq 2d-9/2+6 \leq (\log n)/2+6$ for $d>4$.\\
		Assume finally that $3 \leq k < d/2$, so that $d\geq 7$. In this case, we have
		\[
			n = q^{k(d-k)} \prod_{i=1}^{k}\frac{q^{d-k+i}-1}{q^i-1}=q^{k(d-k)}\frac{q^d-1}{q-1}\cdot\frac{q^{d-1}}{q^2-1}\cdot\frac{q^{d-2}}{q^3-1}\prod_{i=4}^{k}\frac{q^{d-k+i}-1}{q^i-1}\geq q^{3d/2}q^{d+3}=q^{5d/2+3}.
		\]
		In conclusion, $b(G) \leq b_k(G) \leq d/3+8 \leq 5d/4+3/2+6  \leq (\log n)/2+6$. \\
		We now deal with $\Omega_k^2$. By \cite{BGiu}, we see that
		\[
			n = \frac{\prod_{i=d-2k+1}^{d}(q^i-1)}{\prod_{i=1}^k (q^i-1)^2}.
		\]
		Suppose now that $k=1$. Thus, $b(G) \leq b_1(G) \leq d+1$. Moreover,
		\[
			n=\frac{(q^{d-1}-1)(q^d-1)}{(q-1)^2} \geq q^{2d-3},
		\]
		so that $b(G)\leq d+1\leq d-3/2+6 \leq (\log n)/2+6$.\\
		Suppose now that $k=2$, so that $d>4$. Thus, $b(G) \leq d/2+4$ by \cite[Lemma~$4$]{MRD} while
		\[
			n = \frac{(q^{d-3}-1)(q^{d-2}-1)(q^{d-1}-1)(q^{d}-1)}{(q-1)^2(q^2-1)^2} \geq q^{4d-12}.
		\]
		If $d\geq 7$, $b(G) \leq d/2+4 \leq 2d-6 \leq (\log n)/2+6$. For $d=5,6$ we use the exact value of $n$ to see that, for $q=2$, we have that $b(G) \leq (\log n)/2+6$.. The result follows then by noting that $n$ is an increasing function of $q$.\\
		Suppose finally that $3\leq k < d/2$, so that $d>6$. We have
		\[
			n = \frac{\prod_{i=1}^{2k}(q^{i+d-2k}-1)}{\prod_{i=1}^{k}(q^i-1)^2} \geq \frac{\prod_{i=1}^{2k}q^{i+d-2k-1}}{\prod_{i=1}^{k}q^{2i}}=q^{-3k^2+2kd-2k}.
		\]
		Suppose now that $d\geq 16$, and consider the quadric $f(k) = -3k^2+2kd-2k$ with $3 \leq k < d/2$. With a simple computation, it is easy to see that $f(k)$ reaches it minimums for $k=3$, since $d \geq 16$. Thus, $f(k) \geq f(3) = 6d-33$. As before, $b(G) \leq b_k(G) \leq d/k+8 \leq d/3+8$ by Theorem~\ref{basesizeHLM} and Lemma~\ref{lemmaQuotCyc}. Now, $ b(G) \leq d/3+8 \leq  3d-33/2+6 \leq (\log n)/2+6  $ for $d \geq 16$, and so we are done. \\
		If $d< 16$, we compute the minimum point $m_{d,k}$ of the quadric $f$ for each value of $d$ and $k$, and a direct computation shows then that $d/k + 8 \leq f(m_{d,k})/2+6$ for all $q\geq 2$.
	\end{proof}
	\begin{proposition}
		Let $G$ and $H$ be as in case~$\mathrm{\ref{HiHo2}}$, and let $n = [G_0:H]$. Then, $b(G) \leq (\log n)/2+6$.
	\end{proposition}
	\begin{proof}
		In \cite[Section~$7$]{MasS}, it is proved that $G$ admits a base of cardinlality $5$, so that $b(G) \leq 5 \leq (\log n)/2+6$.
	\end{proof}
	\begin{proposition}
		Let $G$ and $H$ as in case~$\mathrm{\ref{HiHo3}}$,  and let $n=[G_0:H]$. Then, $G$ admits a base of cardinality $10$. In particular, $b(G) \leq (\log n)/2+6$.
	\end{proposition}
	\begin{proof}
				There is a geometric description of the domain $\Omega$ of $G$. We briefly recall a few basic facts about trialities, see~\cite[Chapter~8]{PolarSpaces}. Let $V=\mathbb{F}_q^8$ be the orthogonal space associated to $G$. The $4$-dimensional totally singular subspaces fall into two families (usually called the greeks and the latins), where two subspaces $W_1,W_2$ are in the same class if $\dim(W_1)-\dim(W_1\cap W_2)$ is even. Now, $\Omega$ can be identified with the collection of all triples $\{p,\pi_1,\pi_2\}$, where $p$ is a point (that is, a $1$-dimensional totally singular subspace) and $\pi_1$ and $\pi_2$ are solids containing $p$ and of different type. The triality is an automorphism of $G_0$ induced by the maps of the polar space which cyclically permutes $p, \pi_1,\pi_2$.\\
				Without loss of generality we may suppose that the hyperbolic quadric is $$Q=X_1X_2+X_3X_4+X_5X_6+X_7X_8.$$
			Consider the following points of $\Omega$.
			\begin{align*}
				\omega_1 &= \{\langle e_1 \rangle, \langle e_1,e_3,e_5,e_7 \rangle, \langle e_1,e_3,e_5,e_8 \rangle \},
				\omega_2 = \{\langle e_3 \rangle, \langle e_1,e_3,e_5,e_7 \rangle, \langle e_1,e_3,e_5,e_8 \rangle \},\\
				\omega_3 &= \{\langle e_5 \rangle, \langle e_1,e_3,e_5,e_7 \rangle, \langle e_1,e_3,e_5,e_8 \rangle \},
				\omega_4 = \{\langle e_2 \rangle, \langle e_2,e_3,e_5,e_7 \rangle, \langle e_2,e_3,e_5,e_8 \rangle \},\\
				\omega_5 &= \{\langle e_4 \rangle, \langle e_1,e_4,e_5,e_7 \rangle, \langle e_1,e_4,e_5,e_8 \rangle \},
				\omega_6 = \{ \langle e_6 \rangle, \langle e_1,e_3,e_6,e_7\rangle, \langle e_1,e_3,e_6,e_8 \rangle\},\\
				\omega_7 &= \{ \langle e_1+e_3+e_5 \rangle, \langle e_1,e_3,e_5,e_7\rangle, \langle e_1,e_3,e_5,e_8 \rangle \}, \\
				\omega_8 &= \{ \langle e_2+e_3+e_7 \rangle, \langle e_2,e_3,e_5,e_7 \rangle, \langle e_2,e_3,e_6,e_7 \rangle\}, \\
				\omega_9 &= \{ \langle e_2+e_3+e_8 \rangle, \langle e_2,e_3,e_5,e_8\rangle, \langle e_2,e_3,e_6,e_8 \rangle \}.
			\end{align*}
			Observe that, since $\langle e_1,e_3,e_5,e_7 \rangle$  is in common to $\omega_1$ and $\omega_2$, every element of $G_{\omega_1,\omega_2}$ fixes $\langle e_1,e_3,e_5,e_7 \rangle$. This implies that $G_{\omega_1,\omega_2}\le\mathrm{P}\Gamma\mathrm{GO}_8^+(q)$, that is, $G_{\omega_1,\omega_2}$ does not contain trialities. Observe now that if $g \in G_0$ fixes $\langle e_1 \rangle, \dots, \langle e_6 \rangle$, then $g$ is (the projective image of) a diagonal matrix  $$\mathrm{diag}(x_1,x_1^{-1},x_3,x_3^{-1},x_5,x_5^{-1},x_7,x_7^{-1}),$$ for some $x_1,x_3,x_5,x_7 \in \mathbb{F}_q \setminus \{0\}$. This can be seen by imposing that $g$ preserve the quadratic form $Q$. Thus, stabilizing also $\langle e_1+e_3+e_5 \rangle, \langle e_2+e_3+e_7 \rangle, \langle e_2+e_3+e_8 \rangle$, we get a base for the action of $G_0$ on the set of totally singular $1$-subspaces. This argument shows that $(G_0)_{\omega_1,\dots,\omega_9} = \mathrm{Aut}(\mathbb{F}_q)$. By stabilizing also $$\omega_{10} = \{ \langle \alpha e_1 \rangle, \langle \alpha e_1,e_3,e_5,e_5 \rangle, \langle \alpha e_1 ,e_3,e_5,e_8 \rangle\},$$ where $\alpha \in \mathbb{F}_q$ is a generator of the multiplicative group of the field, we get a basis for the action of $G$ on $\Omega$. This shows that $b(G) \leq 10$. However, in this case, we see by \cite{BGiu} that 
			\[
				n = \frac{1}{c}(q+1)^3(q^2+1)^2\left(\frac{q^6-1}{q^2-1}\right),
			\] 
			where $c=2$ if $q$ is odd, and $c=1$ otherwise. In particular, $n$ is an increasing function of $q$ and, for $q=2$, we have
			\[
				10 \leq \frac{1}{2}\log n + 6,
			\]
			and this concludes the proof.
	\end{proof}
	\begin{proposition}
		Let $G$ and $H$ be as in case~$\mathrm{\ref{HiHo4}}$, and let $n = [G_0:H]$. Then, $b(G) \leq (\log n)/2+6$.
	\end{proposition}
	\begin{proof}
		In \cite[Lemma~$8.3$]{MasS} it is proved that $G$ admits a base of cardinality $4$, therefore $b(G) \leq 4 \leq (\log n)/2+6$.
	\end{proof}
	\subsection{Diagonal groups} \label{sec1:sub2}
	\begin{proposition}
		Let $G$ be a primitive group of diagonal type with socle $G_0^k$, for some non-abelian simple group $G_0$, of degree $n=|G_0|^{k-1}$. Then, $ b(G) \leq (\log n)/2+6$.
	\end{proposition}
	\begin{proof}
		Observe that $G$ induces an action by conjugation on the $k$ copies of $G_0$. We denote this group with $P_G$, using the same notation of \cite{Faw}. All the results on the base size of a diagonal group that we will use during this proof are from \cite{Faw}.\\
		It is proved that if $P_G$ does not contain $A_k$, then $b(G) = 2 < (\log n)/2+6$, so we may suppose that $A_k \leq P_G$. If $k=2$, then $b(G) \in \{3,4\}$ and thus $b(G) \leq (\log n)/2+6$.\\		
		We can suppose that $A_k \leq P_G$ and $k > 2$. Then, 
		\[
		b(G) = \left\lceil \frac{\log k}{\log |G_0|} \right\rceil + a_G,
		\]
		where $a_G \in \{1,2\}$. Since $\log n \geq (k-1)\log |G_0|$, to prove that $b(G) \leq (\log n)/2+6$ it suffices to show that 
		\begin{equation}
			\label{eq2}
			\left\lceil \frac{\log k}{\log |G_0|} \right\rceil + a_G  \leq \frac{1}{2}(k-1)\log |G_0|+6.
		\end{equation}
		For the left side of (\ref{eq2}), we have
		\[
		\left\lceil \frac{\log k}{\log |G_0|} \right\rceil + a_G  \leq \frac{\log k}{\log 60} + 3
		\]
		while for the right side 
		\[
		\frac{1}{2}(k-1)\log |G_0|+6 \geq \frac{1}{2}(k-1)\log 60+6.
		\]
		So to prove (\ref{eq2}) it is sufficient to prove that
		\[
		\frac{\log k}{\log 60} + 3 \leq\frac{1}{2}(k-1)\log 60+6,
		\]
		which is equivalent to
		\[
		2\log k -(k-1)(\log 60)^2 \leq 6 \log 60.
		\]
		Now the left side of this inequality is a decreasing function of $k$, and for $k=3$ its value is negative, and this concludes the proof.
	\end{proof}
	\subsection{Affine groups}\label{sec1:sub3}
	\begin{proposition}
		Let $G\leq \mathrm{AGL}_d(q)$, with $q\geq 4$ in its natural action of degree $n=q^d$. Then, $b(G) \leq (\log n)/2+6$.
	\end{proposition}
	\begin{proof}
		We may suppose $G = \mathrm{AGL}_d(q)$. The point stabilizer of the zero vector in $\mathrm{AGL}_d(q)$ is $\mathrm{GL}_d(q)$ acting on the non-zero vectors of $\mathbb{F}_q^d$, and a base for $\mathrm{GL}_d(q)$ is in particular a basis for the vector space. In conclusion, $b(G) \leq b(\mathrm{GL}_d(q)) +1 = d+1$. Now $(\log n)/2+6 \geq (\log 4^d)/2+6 \geq d+6 \geq d+1$.
	\end{proof}
	\subsection{Product action}\label{sec1:sub4}
	\begin{proposition}
		Let $G \leq H \wr \mathrm{Sym}(k)$ with product action on $\Omega = \Gamma^k$, where $H \leq \mathrm{Sym}(\Gamma)$, with $|\Gamma|=n$, is a primitive group  of almost simple type or diagonal type, and $k>1$. Suppose moreover that $G$ is not large-base. Then,$b(G) \leq (\log n^k)/2 + 6$.
	\end{proposition}
	\begin{proof}
		We may suppose that $G = H \wr \mathrm{Sym}(k)$. By \cite{BOW10}, we have that 
		\begin{equation}\label{basesizewrpr}
			b(G) \leq \left \lceil \frac{\lceil \log k \rceil}{\lfloor \log n \rfloor} \right \rceil + b(H,\Gamma).
		\end{equation}

		Now, if $H$ is almost simple, then by \cite[Theorem~$1$]{MRD}, we see that $b(H) \leq 2+\log n$, except for $H=M_{24}$ in its action of degree $24$. Moreover, if $H$ is diagonal, then from \cite[Proposition~$9$]{MRD}, we see that $b(H) \leq \log n \leq \log n + 2$.
		Now, if $n\leq 4$, then $H=\mathrm{Sym}(3)$, or $H=\mathrm{Alt}(4)$ or $H=\mathrm{Sym}(4)$. Thus, using (\ref{basesizewrpr}), a direct computation shows that $b(G) \leq (\log n^k)/2+6$. We can then assume $n > 4$. \\
		Firstly, suppose that $H \neq M_{24}$. Thus,
		\[
			b(G) \leq \frac{\log k +1 }{2}+1 + \log n + 2 \leq \frac{k+1}{2}+\log n + 3.
		\]
		Therefore, in order to establish the inequality $b(G) \leq (\log n)/2+6$, it suffices to show
		\[
			\frac{k+1}{2}+\log n + 3 \leq 6+\frac{1}{2}\log n^k.
		\]
		A routine computation shows that this latter inequality is equivalent to
		\[
			k-5 \leq (k-2)\log n.
		\]
		If $k=2$, then the inequality holds. If $k>2$, the right hand side is an increasing function of $n$, and since $n>4$, it is greater than $2(k-2)> k-5$.\\
		Suppose now that $H = M_{24}$. In this case, $b(H) = 7$, $n=24$, and $\lfloor \log n \rfloor = 4$. Thus, as before, $b(G) \leq (1+\log k)/4+8 \leq (k\log 24)/2+6$.
	\end{proof}
	\begin{proposition}
		Let $G$ be a primitive group of twisted wreath product type of degree $n$. Then, $b(G)\leq (\log n)/2 + 6$.
	\end{proposition}
	\begin{proof}
		From \cite{Pre}, it follows that $G \leq L \wr P$ with the product action and $L$ primitive of diagonal type, and so the result follows from the previous proposition. 
	\end{proof}

	\section{Proof of Theorem~\ref{mainthm:1}}\label{sec2}
		 In this section, we prove Theorem~\ref{mainthm:1}. 
		 Let $G$ be a primitive group of degree $n$, which is not large-base. Suppose moreover that $G$ is neither an almost simple group with socle $G_0$ acting on $\Omega$, with $(G_0,\Omega)$ as in Table \ref{tab:md}, nor a subgroup of $\mathrm{AGL}_d(q)$, with $q=2$ or $q=3$. Then, by Theorem~\ref{mainthm:2}, $$\mu(G)b(G) \leq \frac{1}{2}n\log n + 6 n.$$ Observe that $(n\log n)/2 + 6 n \leq n\log n$ for $n \geq 2^{12}$, meaning that $\mu(G)b(G) \leq n\log n$ for all the primitive groups $G$ of degree $n \geq 4096$. If $G$ is a primitive group of degree $n \leq 4095$, we use the computer system $\mathrm{GAP}$ to see that $b(G)\mu(G) \leq n\log n$, except for the Mathieu group $M_24$ of degree $24$. \\
		 In light of this, it remains to prove Theorem~\ref{mainthm:1} only for those groups which are among the exceptions listed in Theorem~\ref{mainthm:2}. For such a group $G$ of degree $n$, we directly establish the inequality $\mu(G)b(G) \leq n \log n$, depending on the specific group. This requires an explicit analysis for affine groups over the field with $2$ or $3$ elements, for large-base groups, and for almost simple groups with socle and domain as detailed in Table~\ref{tab:md}.\\	
 
	\subsection{Affine groups}\label{sec2:sub1}
	\begin{proposition}
		Let $G \leq \mathrm{AGL}_d(q)$ with $q=2$ or $q=3$ in its natural action on $V = \mathbb{F}_q^d$, and let $n=q^d$. Then, $\mu(G) b(G) \leq n\log n$.
	\end{proposition}
	\begin{proof}
	If $q=3$, $G \leq \mathrm{AGL}_d(3)$, $b(G) \leq d+1$ and $n=3^d$. Observe that, if $d\geq 2$, $b(G) \leq \log n$: indeed, we have that $d+1 \leq d \log 3$ for $d>1/(\log 3-1)\approx 1.7$. 
	If $d=1$, then $G =\mathrm{Alt}(3)$ or $G = \mathrm{Sym}(3)$ and a direct computation shows that $b(G)\mu(G) \leq 3 \log 3$.\\
	Suppose now that $q=2$, so that $G \leq \mathrm{AGL}_d(2)$. In particular, $G = V : H$, where $V =\mathbb{F}_2^d$ and $H$ is an irreducible subgroup of $\mathrm{GL}_d(2)$. Observe that, for any $h \in H\setminus \{1\}$, the set $$\mathrm{fix}(h) = \{v \in V \, : \, v^h = v\}$$ is a subspace of $V$. Thus,
	\[
		\mu(H) = 2^d-\max_{h \in H \setminus \{1\}} |\mathrm{fix}(h)| = 2^d-2^t,
	\]
	for some $0 \leq t < d$. Since $H \leq G$, we have $\mu(G) \leq \mu(H) = 2^d-2^t$. Take now $v_1,\dots,v_t$ to be a basis for the vector space $\mathrm{fix}(\tilde{h})$, where $\tilde{h}\in H$ is the element realizing the minimal degree of $H$ (that is, the element with the largest number of fixed points in $H$). Consider $H_{v_1,\dots,v_t}$. This fixes all the elements of $\mathrm{fix}(\tilde{h})$, since it is a vector space. Take now $v_{t+1}\notin \mathrm{fix}(h)$, and consider $h \in H_{v_1,\dots,v_{t},v_{t+1}}$. This fixes all the points of $\langle v_1,\dots,v_{t+1} \rangle$, and by the maximality of $|\mathrm{fix}(\tilde{h})|$, we have $h = 1$. In other words, the set $\{v_1,\dots,v_{t+1}\}$ is a base for $H$. Now, $b(G) \leq b(H) +1 \leq t+2$. Thus, $\mu(G) b(G) \leq (2^d-2^t)(t+2)$, and it is therefore sufficient to show that
	\[
		(2^d-2^t)(t+2) \leq d2^d = n\log n.
	\]
	If $t = d-1$, then the inequality reduces to $d+1\leq 2d$ which holds. If $t \leq d-2$, then $t+2 \leq d$ and $2^d-2^t \leq 2^d$, and so we are done.
	
	\end{proof}
	
	\subsection{Almost simple groups as in Table \ref{tab:md}}\label{sec2:sub2}
	Here we prove Theorem~\ref{mainthm:1} for almost simple group $G$ having socle $G_0$ acting on the domain $\Omega$, with $(G_0,\Omega)$ listed in Table~\ref{tab:md}. In some cases, it will suffice to focus solely on the base size. In fact, while proving Theorem~\ref{mainthm:1} for the remaining primitive groups, we will also demonstrate this result, which may be of independent interest.
	\begin{proposition}\label{prop:bleqlogn}
		Let $G$ be a primitive permutation group of almost simple type having socle $G_0$ acting on $\Omega$, where $(G_0,\Omega)$ are as in Table~\ref{tab:md2}, and let $n=|\Omega|$. Then, we have $b(G) \leq \log n$.
	\end{proposition}	
	\begin{table}
		\[
		\begin{array}{lllll} \hline
			G_0 & \Omega & \mbox{Conditions} \\ \hline
			{\rm PSL}_{d}(3) & P_1  & d \geq 4  \\
			\hline  
			{\rm PSp}_d(3) & P_1 & d \geq 6 \\
			\hline 
			\mathrm{P}\Omega_d^+(2) & S_1 & d \geq 8 \\ \hline
			\mathrm{P}\Omega_d^\varepsilon(3) & S_1 & d \geq 8 		\\	
			& N_1 &  d \geq 8 		\\	
			\hline 
			\mathrm{P}\Omega_d(3) & S_1 & d \geq 7\\
			& N_{d-1} & d \geq 7\\
			\hline
		\end{array}
		\]
		\caption{Groups of Proposition~\ref{prop:bleqlogn} satisfying $b(G)\leq \log n$}
		\label{tab:md2}
	\end{table}
	\subsubsection{Linear groups}
	\begin{proposition}
		Let $G$ be an almost simple group having socle $G_0 = \mathrm{PSL}_d(q)$, with $d\geq 3$ and $q=2$ or $q=3$, acting on $P_1$, and let $n=|P_1|$. Then, $\mu(G)b(G) \leq n\log n$.
	\end{proposition} 
	\begin{proof}
		Suppose first that $G_0 = \mathrm{PSL}_d(2)$. Then, $b(G) \leq b(G_0)+1 \leq d+1$. Moreover, from \cite[Table~$2$]{BG}, we see that $\mu(G) = 2^{d-1}$. Thus, $\mu(G) b(G) \leq 2^{d-1}(d+1)$, and $2^{d-1}(d+1) \leq n\log n$ if and only if
		\[
			\frac{d+1}{\log(2^d-1)}\leq \frac{2^d-1}{2^{d-1}}.
		\]
		To see this, just observe that the left hand side is less then $3/2$ while the right hand side is greater then $3/2$. \\
		Suppose now that $q=3$. Here, $b(G) \leq d+3$, while $n=(3^d-1)/2$. Now $\log n \geq \log(3^d-3^{d-1})-1 = \log(3^{d-1}(3-1))-1 = (d-1)\log 3 $. Observe now that if $d \geq 8$, then $b(G) \leq d+3 \leq (d-1)\log 3  \leq \log n$.  For $d=3,4,5,6,7$, we have verified that $b(G) \leq \log n$ using $\mathrm{GAP}$ \cite{GAP}.
	\end{proof}
	
	\subsubsection{Symplectic groups}
	\begin{proposition}
		Let $G$ be an almost simple group having socle $G_0 = \mathrm{PSp}_d(q)$, with $d\geq 4$ even and $q=2$ or $q=3$, acting on $S_1$, and let $n=|S_1|$. Then, $\mu(G)b(G) \leq n\log n$.
	\end{proposition} 
	\begin{proof}
		If $q=2$, then $b(G) \leq d+2$ by \cite[Lemma~$3$]{MRD} and Lemma~\ref{lemmaQuotCyc}, while $\mu(G) = 2^{d-1}$ by \cite[Table~$2$]{BG}. Moreover, $n = 2^d-1$, so the inequality $\mu(G) b(G) \leq n \log n$ follows from the case $\mathrm{PSL}_d(2)$. \\
		Suppose now that $q=3$. As before, $b(G)\leq d+3$ and $n = (3^d-1)/2$, and with the same computation for the linear case, we get $b(G)\leq \log n$ for $d \geq 8$. For $d=4,6$, we have checked that $b(G) \leq \log n$ using GAP.
	\end{proof}
	\begin{proposition} \label{SpGO-}		
		Let $d>4$ even and let $G = \mathrm{Sp}_d(2)$ acting on the coset of $\mathrm{GO}_d^\varepsilon(2)$. Setting $n = [G : \mathrm{GO}_d^\varepsilon(2)]$, we have, $\mu(G) b(G) \leq  n \log n$.
	\end{proposition}
	\begin{proof}
		By \cite[Proposition~$7$]{MRD}, we have that $b(G) \leq 2+\log n$. Moreover, from \cite[Theorem~$4$]{BG}, we have that $\mu(G) \leq 2/3n$. Thus, $\mu(G) b(G) \leq (2/3)n(2+\log n)$, and $(2/3)n(2+\log n)\leq n \log n$ for $n\geq 16$. However, in this case,
		\[
			n = 2^{\frac{d}{2}-1}(2^{\frac{d}{2}}+\varepsilon) \geq 2^{3-1}(2^3+\varepsilon) > 16,
		\]
		since $d \geq 6$.
	\end{proof}
	\subsubsection{Orthogonal groups of even dimension}
	\begin{proposition}
		Let $G$ be an almost simple group having socle $G_0 = \mathrm{P}\Omega_d^\varepsilon(q)$, with $d\geq 8$ even and $q=2$ or $q=3$, acting on $S_1$, and let $n=|S_1|$. Then, $\mu(G)b(G) \leq n\log n$.
	\end{proposition} 
	\begin{proof}
	Suppose first that $q=2$ and $\varepsilon = +$. In this case, we have that
	\[
		b(G) \leq \log n.
	\]
	Indeed, by Lemma~\ref{lemmaQuotCyc}, we get $b(G) \leq d-1$, while $n = (2^{d/2}-1)(2^{d/2-1}+1)$. Thus, $b(G) \leq \log n$ if and only if
	\[
		2^{d-1} \leq (2^{\frac{d}{2}}-1)(2^{\frac{d}{2}-1}+1) = 2^{d-1} + 2^{\frac{d}{2}}-2^{\frac{d}{2}-1}-1
	\]
	and this holds if and only if $2^{d/2}-2^{d/2-1}-1 \geq 0$ which is true for $d\geq 4$.\\
	Suppose now that $q=2$ and $\varepsilon = -$. In this case, it is easy to see that it is no longer true that $b(G) \leq \log n$, where
	\[
		n = (2^\frac{d}{2}+1)(2^{\frac{d}{2}-1}-1).
	\]
	So, we need to discuss also the minimal degree. We use the quadratic elliptic form given by
	\[
		Q = X_1^2+X_1X_2+X_2^2+\sum_{i=3}^{d-1}X_iX_{i+1}.
	\]
	Consider now the (image of) the matrix
	\[
		g = 
		\begin{bmatrix}
			A & 0 \\
			0 & I_{d-2}
		\end{bmatrix},
	\]
	where 
	\[
		A = \begin{bmatrix}
			1 & 1 \\
			1 & 0
		\end{bmatrix}.
	\]
	Observe that $g \in \mathrm{P}\Omega_d^-(2)$, since it is the product of the reflections trough $e_1$ and $e_2$ respectively. Observe now that $g$ fixes all the totally isotropic $1$-subspaces of $W=\langle e_3,\dots,e_d \rangle$. However, if we restrict $Q$ to this subspace, we recover a quadratic hyperbolic form. Thus, the number of totally isotropic $1$-dimensional subspaces of $W$ is
	\[
		\tilde{n} = (2^{\frac{d}{2}-1}-1)(2^{\frac{d}{2}-2}+1).
	\]
	In particular, since $g$ moves $n-\tilde{n}$ 
	points, we have that $$\mu(G) \leq \mu(\mathrm{P\Omega_d^-(2)}) \leq |\mathrm{supp}(g)| = n-\tilde{n} = 3(2^{d-3}-2^{\frac{d}{2}-2}).$$
	Moreover, we have that $b(G)\leq d-1$. In particular, $\mu(G) b(G) \leq 3(2^{d-3}-2^{d/2-2})(d-1)$. Therefore, it suffices to show that
	\begin{equation}\label{ineqOrt1}
		3(d-1)(2^{d-3}-2^{\frac{d}{2}-2}) \leq (2^\frac{d}{2}+1)(2^{\frac{d}{2}-1}-1) \log((2^\frac{d}{2}+1)(2^{\frac{d}{2}-1}-1)). 
	\end{equation}	
	Firstly, observe that $n \geq 2^{d-2}$. Indeed, 
	\[
		n=(2^\frac{d}{2}+1)(2^{\frac{d}{2}-1}-1) \geq 2^\frac{d}{2}(2^{\frac{d}{2}-1}-2^{\frac{d}{2}-2}) = 2^\frac{d}{2}(2^{\frac{d}{2}-2}(2-1)) = 2^{d-2}.
	\]	
	Thus, the inequality~\ref{ineqOrt1} reduces to
	\[
		3 \frac{d-1}{d-2}(2^{d-3}-2^{\frac{d}{2}-2}) \leq (2^\frac{d}{2}+1)(2^{\frac{d}{2}-1}-1).
	\]
	Now $3(d-1)/(d-2) \leq 4$, and hence
	\[
		3 \frac{d-1}{d-2}(2^{d-3}-2^{\frac{d}{2}-2}) \leq 2^{d-1}-2^{\frac{d}{2}}.
	\]
	Now it is trivial to see that $2^{d-1}-2^{d/2} \leq (2^{d/2}+1)(2^{d/2-1}-1)$, and so we are done.\\
	Suppose now that $q=3$. Since $\mathrm{Out}(\mathrm{P}\Omega_d^\varepsilon(3))$ has a subnormal series of length at most $2$ (depending on the parity of $d/2$), $b(G) \leq d$. We show that $b(G) \leq \log n$ in both cases. Suppose that $\varepsilon=+$. Thus, $n = (3^{d/2}-1)(3^{d/2-1}+1)/2$, and $d+1 \leq \log n$ if and only if 
	\[
		d+1 \leq \log ((3^\frac{d}{2}-1)(3^{\frac{d}{2}-1}+1)).
	\]
	But we have
	\[
	\log ((3^\frac{d}{2}-1)(3^{\frac{d}{2}-1}+1)) \geq \log(3^{\frac{d}{2}-1}(3^\frac{d}{2}-1)) \geq \log(3^{\frac{d}{2}-1}(3^{\frac{d}{2}}-3^{\frac{d}{2}-1}))=(d-1)\log 3+1,
	\]
 	So the claim follows from $d \leq (d-1)\log 3$, which holds for $d \geq 3$. The case $\varepsilon=-$ is similar.
	
	\end{proof}
	\begin{proposition}
		Let $G$ be an almost simple group having socle $G_0 = \mathrm{P}\Omega_d^\varepsilon(q)$, with $d\geq 8$ even and $q=2$ or $q=3$, acting on $N_1$, and let $n=|N_1|$. Then, $\mu(G)b(G) \leq n\log n$.
	\end{proposition} 
	\begin{proof}
		Suppose that $q=2$. We handle firstly with the case $\varepsilon = -$. Observe that $n=2^{d/2-1}(2^{d/2}+1)$. We again use the quadratic elliptic form $Q = X_1^2+X_1X_2+X_2^2+\sum_{i=3}^{d-1}X_iX_{i+1}$. Take again the element of $G_0$ given by
			\[
		g = 
		\begin{bmatrix}
			A & 0 \\
			0 & I_{d-2}
		\end{bmatrix},
		\]
		where 
		\[
		A = \begin{bmatrix}
			1 & 1 \\
			1 & 0
		\end{bmatrix}.
		\]
		Thus, $g$ fixes all the $1$-singular non degenerate subspaces of $W=\langle e_3,\dots,e_d \rangle$. Observe moreover that $Q_{|W}$ is an hyperbolic quadratic form. Thus, $g$ fixes at least $\tilde{n} = 2^{d/2-2}(2^{d/2-1}-1)$ points. This shows that
		\[
			\mu(G) \leq \mu(G_0) \leq |\mathrm{supp}|(g) \leq n-\tilde{n} = 3.2^{d-3}+2^{\frac{d}{2}-1}.
		\]
		Since $b(G) \leq d$, the inequality $b(G)\mu(G) \leq n \log n$ holds if
		\[
			d( 3.2^{d-3}+2^{\frac{d}{2}-1}) \leq 2^{\frac{d}{2}-1}(2^{\frac{d}{2}}+1)\log(2^{\frac{d}{2}-1}(2^{\frac{d}{2}}+1)).
		\]
		We have that $\log n \geq \log(2^{d/2-1}2^{d/2})=d-1$. Then, the inequality $\mu(G)b(G) \leq n\log n$ holds if
		\[
			\frac{d}{d-1}( 3.2^{d-3}+2^{\frac{d}{2}-1}) \leq 2^{\frac{d}{2}-1}(2^{\frac{d}{2}}+1).
		\]
		Since $d\geq 8$, we have $d/(d-1) \leq 8/7$, and thus the inequality reduces to
		\[
			3.2^d+2^{\frac{d}{2}+2}\leq 7.2^{d-1}+7.2^{\frac{d}{2}-1},
		\]
		which is $2^{d-1}\geq 2^{d/2-1}$, which holds for all positive $d$.
		\\
		Suppose now that $(q,\varepsilon)=(2,+)$. We use the hyperbolic quadratic form 
		\[
			Q = \sum_{i=1}^{d-1} X_iX_{i+1}.
		\]
		Take $v_1 = e_1+e_2$ and $v_2 = e_3+e_4$. Observe that both $v_1$ and $v_2$ are non-singular, and hence we can define the reflections $r_{v_1}$ and $r_{v_2}$. Taking their product, we obtain the element $g$ of $\mathrm{P}\Omega_d^+(2)$ given by the image of the matrix
		\[
			\begin{bmatrix}
				A & 0 \\
				0 & I_{d-4},
			\end{bmatrix}
		\]
		where
		\[
			A = 
			\begin{bmatrix}
				0 & 1 & 0 & 0 \\
				1 & 0 & 0 & 0 \\
				0 & 0 & 0 & 1 \\
				0 & 0 & 1 & 0 
			\end{bmatrix}.
		\]
		Observe now that $g$ fixes all the non degenerate $1$-dimensional subspaces of $W=\langle v_1,v_2,e_5,\dots,e_d \rangle$. If now we restrict $Q$ to $W$, we obtain a quadratic form of elliptic type. Hence, $g$ fixes $2^{d/2-2}(2^{d/2-1}+1)$ points. This shows that
		\[
			\mu(G) \leq \mu(G_0) \leq |\mathrm{supp}(g)| = 2^{\frac{d}{2}-1}(2^{\frac{d}{2}}-1)-2^{\frac{d}{2}-2}(2^{\frac{d}{2}-1}+1) = 3(2^{d-3}-2^{\frac{d}{2}-2}).
		\]
		Now we have that $b(G) \leq b(G_0)+2 \leq d$, and hence to show that $b(G)\mu(G) \leq n \log n$ it is sufficient to show
		\[
			3d(2^{d-3}-2^{\frac{d}{2}-2}) \leq 2^{\frac{d}{2}-1}(2^{\frac{d}{2}}-1)\log(2^{\frac{d}{2}-1}(2^{\frac{d}{2}}-1)).
		\]
		Now observe that
		\[
			\log(2^{\frac{d}{2}-1}(2^{\frac{d}{2}}-1)) \geq \log(2^{\frac{d}{2}-1}(2^{\frac{d}{2}}-2^{\frac{d}{2}-1})) = d-2,
		\]
		hence it suffices to show that 
		\[
			3 \frac{d}{d-2}(2^{d-3}-2^{\frac{d}{2}-2}) \leq 2^{\frac{d}{2}-1}(2^{\frac{d}{2}}-1).
		\]
		Now since $d \geq 8$, $3d/(d-2)\leq 4$, so that the inequality reduces to 
		\[
			2^{d-1}-2^{\frac{d}{2}} \leq 2^{d-1}-2^{\frac{d}{2}-1},
		\]
		which is trivially true.
		\\
		Suppose now that $q=3$. Here the two cases for $\varepsilon$ are very similar, and we report only the case $\varepsilon = +$. We have that $m = 3^{d/2-1}(3^{d/2}-1)/2$, and $b(G) \leq d$. We show that $b(G) \leq \log n$. Indeed, this is equivalent to $d+2 \leq \log (3^{d/2-1}(3^{d/2}-1))$. But we have that 
		\[
			\log (3^{\frac{d}{2}-1}(3^{\frac{d}{2}}-1)) \geq \log (3^{\frac{d}{2}-1}(3^{\frac{d}{d}}-3^{\frac{d}{2}-1})) = \log(2.3^{d-2}) = (d-2)\log3 +1.
		\]
		Thus it suffices to show $d+1 \leq (d-2)\log 3+1$, which is true for $d\geq (2\log3+1)/(\log 3-1) \approx 8.83 >9$. For $d=8$, we checked the validity on the computer system GAP.
	\end{proof}
	\subsubsection{Orthogonal groups in odd dimension}
	\begin{proposition}
		Let $G$ be an almost simple group with socle $G_0 = \mathrm{P}\Omega_d(3)$, with $d\geq 7$ odd, acting on $S_1$, and let $n=|S_1|$. Then, $\mu(G) b(G) \leq n \log n$.
	\end{proposition} 
	\begin{proof}
		In this case, we have that $b(G) \leq d+1$, and $n = (3^{d-1}-1)/2$, and it is trivial to see that $d+1 \leq \log n$ for $d \geq 7$.
	\end{proof}
	\begin{proposition}
		Let $G$ be an almost simple group with socle $G_0 = \mathrm{P}\Omega_d(3)$, with $d\geq 7$ odd, acting on $N_1^\epsilon$, and let $n=|N_1^\epsilon|$. Then, $\mu(G) b(G) \leq n \log n$.
	\end{proposition} 
	\begin{proof}
		Also in this case we have $b(G) \leq \log n$. Indeed, it is shown in \cite[Proposition~$3.7$]{MRD} that $b(G)\leq d$ and $n \geq (1/4)3^{d-1}$, so that $\log n \geq (d-1)\log 3 -2$, so it suffices to show that $(d-1)\log 3-2 \geq d$, and this holds for 
		\[
		d \geq \frac{\log 3 +2}{\log 3-1} \approx6.12.
		\]		
	\end{proof}

	\subsection{Large-base groups}\label{sec2:sub3}
	In this section, we write $S_{m,k}$ (or $A_{m,k}$) to denote the symmetric group $S_m$ (or alternating group $A_m$) in its action on the collection of $k$-subsets of $\{1,\dots,m\}$. In particular, since $S_{m,k} = S_{m,m-k}$, we may suppose that $1\leq k \leq m/2$. \\
	Let now  $G$ be a large-base group, that is $(A_{m,k})^r \trianglelefteq G \leq S_{m,k} \wr S_r$, and the wreath product is equipped with the product action of degree $n = {m \choose k}^r$. Since $(A_{m,k})^r \trianglelefteq G$, then $\mu(G) \leq \mu((A_{m,k})^r)$. We need two basic lemmas.
	\begin{lemma}\label{lemmaAltSubsets}
		Let $m>2$ and $1 \leq k \leq m/2$. Then,
		\[
			\mu(A_{m,k}) \leq 3 {m-2 \choose k-1}.
		\]
	\end{lemma}
	\begin{proof}
		Take the $3$-cycle $(123)$. This moves set of the form $\{a,x_1,\dots,x_{k-1}\}$ or set of the form $\{a,b,x_1,\dots,x_{k-2}\}$ for $a,b \in \{1,2,3\}$ and $x_i \notin \{1,2,3\}$, for $i=1,\dots,k-1$. In the first case, we have $3$ choices for $a$, and ${m-3 \choose k-1}$ choices for $x_1,\dots,x_{k-1}$, and in the second case we have $3$ choices for $a$ and $b$, and ${m-3 \choose k-2}$ choices for $x_1,\dots,x_{k-2}$. In total, we have 
		\[
			3{m-3 \choose k-1}+3{m-3 \choose k-2} = 3{m-2 \choose k-1}. 
		\]
		Now the conclusion follows by $\mu(A_{m,k},\Omega_k) \leq |\mathrm{supp}(x)|$ for any non trivial $x \in A_m$.
	\end{proof}
	\begin{lemma}\label{lemmaMinDegWrPr}
		Let $H \leq \mathrm{Sym}(\Gamma)$ and $K \leq \mathrm{Sym}(\Lambda)$ be two permutation groups, and consider $G = H \wr K$ in its product action on $\Omega = \Gamma^k$, where $k = |\Lambda|$. Then,
		\[
			\mu(H\wr K,\Omega) \leq \mu(H,\Gamma)|\Gamma|^{k-1}.
		\] 
	\end{lemma}
	\begin{proof}
		Take $h \in H$ such that $|\mathrm{supp}(h)|=\mu(H,\Lambda)$, and consider $x=(h,1,\dots,1) \in G$. This moves points of $\Omega$ of the form $(x,y_1,\dots,y_{k-1})$, for $x \in \mathrm{supp}(h)$, $y_i \in \Gamma$ for $i=i,\dots,k-1$, and the result follows.
	\end{proof}
	We are now ready to prove Theorem \ref{mainthm:1} for large-base groups.
	\begin{proposition}
		Let $G$ be a large-base group, so that $(A_{m,k})^r \trianglelefteq G \leq S_{m,k} \wr S_r$, with $1 \leq k \leq m/2$. Setting $n = {m \choose k}^r$, we have
		\[
			\mu(G)b(G) \leq n \log n
		\]
	\end{proposition}
	\begin{proof}
	
		Denote with $\Omega_k$ the collection of $k$-subsets of $\{1,\dots,m\}$. By Lemma~\ref{lemmaAltSubsets} and Lemma~\ref{lemmaMinDegWrPr}, we have
		\begin{equation}\label{ineqmindeg}
			\mu(G) \leq \mu((A_{m,r})^r) \leq \mu(A_{m,r}) |\Omega_k|^{r-1} \leq3 \frac{{m-2 \choose k-1}}{{m \choose k}} n = 3 \frac{k(m-k)}{m(m-1)}n.
		\end{equation} 
		
		Firstly, suppose that $m\leq 20$ and $r \leq 40$. In this case, we can bound the base size of $G$ using \cite{BOW10}, and obtain
		\[
			b(G) \leq \left \lceil \frac{\lceil \log r \rceil}{\lfloor \log m \rfloor} \right \rceil + b(S_{m,k}).
		\] 
		We can now use the main result of \cite{DVRD} to compute, with $\mathrm{GAP}$, the exact value of $b(S_{m,k})$. Using these two upper bound for the minimal degree and the base size, we checked with $\mathrm{GAP}$ that, for all the values of $m \leq 20$, $r \leq 40$ and $1\leq k \leq 10$, we have
		\[
			3 \frac{k(m-k)}{m(m-1)}n\left(\left \lceil \frac{\lceil \log r \rceil}{\lfloor \log m \rfloor} \right \rceil + b(S_{m,k})\right) \leq n \log n.
		\]
		We can now suppose that $m \geq 20$, $r \geq 40$. We now need to use a different upper bound for the base size. In particular, from \cite{HLM}, we have that $b(G) \leq 2\log |G|/\log n + 22\leq 24 \log |G| / \log n$. But $G \leq S_{m,k} \wr S_r$, so that $|G| \leq m!^r r! \leq m^{mr} r^r$, and so
		\begin{equation}\label{ineqbasesize}
				b(G) \leq 24\frac{mr\log m + r\log r }{\log n}.
		\end{equation}
		Thus, combining \ref{ineqmindeg} and \ref{ineqbasesize}, we have
		\[
			\mu(G)b(G) \leq  72 \frac{k(m-k)}{m(m-1)}n\left(\frac{mr\log m + r\log r }{\log n}\right).
		\]
		Observe that $(m-k)/(m-1) \leq 1$, so that
		\[
		 	72 \frac{k(m-k)}{m(m-1)}n\left(\frac{mr\log m + r\log r }{\log n}\right) \leq 72n\frac{k}{m}\left(\frac{mr\log m + r\log r }{\log n}\right).
		\]
		Thus, the inequality $\mu(G) b(G) \leq n\log n$ holds if
		\[
			 72n\frac{k}{m}\left(\frac{mr\log m + r\log r }{\log n}\right) \leq n\log n,
		\]
		which is, after some computations
		\begin{equation}\label{ineqLargBase1}
			72k\left(r\log m + r \frac{\log r}{m}\right) \leq \log^2 n.
		\end{equation}
		Now $n = {m \choose k}^r \geq (m/k)^{rk}$, so that $\log^2 n \geq r^2k^2 \log^2(m/k)$. To prove (\ref{ineqLargBase1}), and therefore to conclude the proof, it is sufficient to show that
		\[
			\log m + \frac{\log r}{m} \leq \frac{rk}{72}\log^2\left(\frac{m}{k}\right).
		\]
		To do this, consider the function 
		\[
			f(m,r,k) = \frac{rk}{72}\log^2\left(\frac{m}{k}\right) - \log m - \frac{\log r}{m},
		\]
		restricted to $m \geq 20$, $r \geq 40$, $1 \leq k \leq m/2$. We now show that $f$ is an increasing function in all the three variables, and then we conclude by noting that $f(20,40,1) \geq 0$.\\
		From now on, we consider the discrete variables $m,k,r$ as continuous, and we write $\ln$ to indicate the logarithm in base $e$. Setting $x=(m,k,r)$, we have
		\[
			\frac{\partial f(x)}{\partial m} = \frac{rk}{36m\ln^22}\ln \left(\frac{m}{k}\right) + \frac{\log r}{m^2}-\frac{1}{m\ln(2)}. 
		\]
		Thus, $\partial f(x)/\partial m \geq 0$ if and only if
		\begin{equation}\label{eqpart1}
			\frac{rkm}{36n\ln^22}\ln \left(\frac{m}{k}\right)-\frac{m}{\ln 2} \geq- \log r.
		\end{equation}		
		Now, $r \geq 40$, so that $-\log r < 0$. Thus, to show (\ref{eqpart1}) it is sufficient to show that
		\[
		\frac{rkm}{36m\ln^22}\ln \left(\frac{m}{k}\right)-\frac{m}{\ln 2} \geq 0,
		\]
		which is
		\[
			\frac{m}{\ln 2}\left(\frac{rk}{36\ln2}\ln\left(\frac{m}{k}\right)-1\right) \geq 0.
		\] 
		The first factor is of course positive. For the second one, we must show that 
		\[
			\ln \left(\frac{m}{k}\right) \geq \frac{36 \ln 2}{rk}.
		\]
		Now, $r \geq 40$ and $k \geq 1$, so that $(36 \ln 2)/rk \leq (36 \ln 2)/40$. But $k \leq m/2$, so that $m/k \geq 2$, and thus $\ln(m/k)\geq \ln 2$. Now the claim follows from $\ln2 \geq (36 \ln 2)/40$. Thus, $f$ is an increasing function on $m$.\\
		Consider now the variable $r$. We have
		\[
			\frac{\partial f(x)}{\partial r} = \frac{k}{72}\log^2\left(\frac{m}{k}\right)-\frac{1}{rm\ln2},
		\]
		so that $\partial f(x)/\partial r \geq 0$ if and only if
		\[
			km\log^2\left(\frac{m}{k}\right) \geq \frac{72}{r \ln 2}.
		\]
		Since $r \geq 40$, $72/(r\ln2)\leq 72/(40\ln 2)$. Moreover, since $m/k \geq 2$, $km \log^2(m/k) \geq km$, thus it is sufficient to show that 
		\[
			km \geq \frac{72}{40 \ln 2},
		\]
		which is trivially true.\\
		Finally, consider the variable $k$, and remark that $1 \leq k \leq m/2$. We have
		\[
			\frac{\partial f(x)}{\partial k} = \frac{r}{72}\log^2\left(\frac{m}{k}\right) - \frac{r}{36\ln^22}\ln\left(\frac{1}{k}\right).
		\]
		Thus, $\partial f(x)/\partial r\geq 0$ if and only if
		\[
			\log^2\left(\frac{n}{k}\right)\geq \frac{2}{\ln 2}\ln\left(\frac{1}{k}\right).
		\]
		But $\ln(1/k) \leq 0$, while $\log^2(n/k)\geq 1$, and so we are done.
	\end{proof}

\end{document}